\documentclass[12pt,reqno]{amsart}
\usepackage{latexsym}
\usepackage{amssymb}
\usepackage{mathrsfs}
\usepackage{amsmath}
\usepackage{comment}
\usepackage{fancybox,color}
\usepackage{enumerate}
\usepackage[latin1]{inputenc}
\usepackage{soul} 
\usepackage[colorlinks=true, linkcolor=blue, citecolor=blue]{hyperref}

{\catcode`p =12 \catcode`t =12 \gdef\eeaa#1pt{#1}}      %
\def\accentadjtext#1{\setbox0\hbox{$#1$}\kern   %
                \expandafter\eeaa\the\fontdimen1\textfont1 \ht0 }
\def\accentadjscript#1{\setbox0\hbox{$#1$}\kern %
                \expandafter\eeaa\the\fontdimen1\scriptfont1 \ht0 }
\def\accentadjscriptscript#1{\setbox0\hbox{$#1$}\kern   %
                \expandafter\eeaa\the\fontdimen1\scriptscriptfont1 \ht0 }
\def\accentadjtextback#1{\setbox0\hbox{$#1$}\kern       %
                -\expandafter\eeaa\the\fontdimen1\textfont1 \ht0 }
\def\accentadjscriptback#1{\setbox0\hbox{$#1$}\kern     %
                -\expandafter\eeaa\the\fontdimen1\scriptfont1 \ht0 }
\def\accentadjscriptscriptback#1{\setbox0\hbox{$#1$}\kern %
                -\expandafter\eeaa\the\fontdimen1\scriptscriptfont1 \ht0 }
\def\itoverline#1{{\mathsurround0pt\mathchoice
        {\rlap{$\accentadjtext{\displaystyle #1}
                \accentadjtext{\vrule height1.593pt}
                \overline{\phantom{\displaystyle #1}
                \accentadjtextback{\displaystyle #1}}$}{#1}}
        {\rlap{$\accentadjtext{\textstyle #1}
                \accentadjtext{\vrule height1.593pt}
                \overline{\phantom{\textstyle #1}
                \accentadjtextback{\textstyle #1}}$}{#1}}
        {\rlap{$\accentadjscript{\scriptstyle #1}
                \accentadjscript{\vrule height1.593pt}
                \overline{\phantom{\scriptstyle #1}
                \accentadjscriptback{\scriptstyle #1}}$}{#1}}
        {\rlap{$\accentadjscriptscript{\scriptscriptstyle #1}
                \accentadjscriptscript{\vrule height1.593pt}
                \overline{\phantom{\scriptscriptstyle #1}
                \accentadjscriptscriptback{\scriptscriptstyle #1}}$}{#1}}}}

\newcommand{\iol}{\itoverline}

\newcommand{\ch}[1]{{\mbox{\raise 1pt\hbox{\large$\chi$}}}_{\lower 1pt\hbox{$\scriptstyle #1$}}}

\usepackage[left=2.1cm,top=2.25cm,right=2.1cm]{geometry}

\def\1{\raisebox{2pt}{\rm{$\chi$}}}

\newtheorem{theorem}{Theorem}[section]

\newtheorem{lemma}[theorem]{Lemma}

\theoremstyle{definition}
\newtheorem{definition}[theorem]{Definition}
\newtheorem{remark}[theorem]{Remark}
\newtheorem{example}[theorem]{Example}

\DeclareFontFamily{U}{mathx}{}
\DeclareFontShape{U}{mathx}{m}{n}{<-> mathx10}{}
\DeclareSymbolFont{mathx}{U}{mathx}{m}{n}
\DeclareMathAccent{\widehat}{0}{mathx}{"70}
\DeclareMathAccent{\widecheck}{0}{mathx}{"71}

\newcommand{\R}{{\mathbb R}}

\newcommand{\N}{{\mathbb N}}
\newcommand{\Z}{{\mathbb Z}}

\newcommand{\Ha}{{\mathcal H}}

\newcommand\cp{\operatorname{cap}}
\newcommand\diam{\operatorname{diam}}

\newcommand{\eps}{{\varepsilon}}
\def\1{\raisebox{2pt}{\rm{$\chi$}}}

\def\vint_#1{\mathchoice%
        {\mathop{\kern 0.2em\vrule width 0.6em height 0.69678ex depth -0.58065ex
                \kern -0.8em \intop}\nolimits_{\kern -0.4em#1}}%
        {\mathop{\kern 0.1em\vrule width 0.5em height 0.69678ex depth -0.60387ex
                \kern -0.6em \intop}\nolimits_{#1}}%
        {\mathop{\kern 0.1em\vrule width 0.5em height 0.69678ex
            depth -0.60387ex
                \kern -0.6em \intop}\nolimits_{#1}}%
        {\mathop{\kern 0.1em\vrule width 0.5em height 0.69678ex depth -0.60387ex
                \kern -0.6em \intop}\nolimits_{#1}}}
\def\vintslides_#1{\mathchoice%
        {\mathop{\kern 0.1em\vrule width 0.5em height 0.697ex depth -0.581ex
                \kern -0.6em \intop}\nolimits_{\kern -0.4em#1}}%
        {\mathop{\kern 0.1em\vrule width 0.3em height 0.697ex depth -0.604ex
                \kern -0.4em \intop}\nolimits_{#1}}%
        {\mathop{\kern 0.1em\vrule width 0.3em height 0.697ex depth -0.604ex
                \kern -0.4em \intop}\nolimits_{#1}}%
        {\mathop{\kern 0.1em\vrule width 0.3em height 0.697ex depth -0.604ex
                \kern -0.4em \intop}\nolimits_{#1}}}

\newcommand{\intav}{\vint}

\newcommand{\Om}{\Omega}

\newcommand{\dist}{\operatorname{dist}}

\title[Fractional Hardy inequalities and capacity density]{Fractional Hardy inequalities and capacity density}

\author[L. Ihnatsyeva]{Lizaveta Ihnatsyeva}   %
\address[L.I.]{Department of Mathematics, Kansas State University, Manhattan, KS 66506, USA}
\email{ihnatsyeva@math.ksu.edu}

\author[K. Mohanta]{Kaushik Mohanta}   %
\address[K.M.]{Department of Mathematics and Statistics, P.O. Box 35, FI-40014 University of Jyvaskyla, Finland}
\email{kaushik.k.mohanta@jyu.fi}

\author[A. V. V\"ah\"akangas]{Antti V. V\"ah\"akangas}
\address[A.V.V.]{Department of Mathematics and Statistics, P.O. Box 35, FI-40014 University of Jyvaskyla, Finland}
 \email{antti.vahakangas@iki.fi}

\makeatletter
\@namedef{subjclassname@2020}{{\mdseries 2020} Mathematics Subject Classification}
\makeatother



\begin{document}

\begin{abstract}
We prove that a pointwise fractional Hardy inequality implies a fractional Hardy inequality,  defined via a Gagliardo-type seminorm.  The proof consists of two main parts. The first one is to characterize the pointwise fractional Hardy inequality in terms of  a fractional capacity density condition. The second part is to show the deep open-endedness or self-improvement property of the fractional capacity density, which we accomplish  in the setting of a complete geodesic space equipped with a doubling measure. These results are new already in the standard Euclidean setting. 
\end{abstract}
\maketitle

\section{Introduction}

This paper brings together the pointwise fractional Hardy inequalities introduced in
 \cite{MR4649157} and the recent developments on self-improvement of nonlocal capacity density conditions from \cite{MR946438,MR4478471,CILV,LMV}.
The motivation for our nonlocal problems comes from their local analogues.
 Recall that a closed set $E$ in a metric measure space, equipped with a doubling measure, supports a pointwise $p$-Hardy inequality
if there exists a constant $C_{\mathrm{H}}$ such that inequality
\begin{equation}\label{eq.hardy_intro}
|u(x)| \leq  C_{\mathrm{H}}\, \dist(x,E) (M_{2\dist(x,E)}(g^p)(x))^{1/p}\,
\end{equation}
holds for all 
 $x\in X\setminus E$, whenever $u$ is a Lipschitz function on $X$ such that $u=0$ in  $E$ and $g$ is 
 an upper gradient of $u$; we refer to
\eqref{e.max_funct_def_restricted} for  the definition
of the restricted maximal operator $M_{2\dist(x,E)}$.
The pointwise $p$-Hardy inequality in $\R^n$ with $g=\lvert \nabla u\rvert$ was first independently studied 
by Kinnunen--Martio in \cite{MR1470421} and by Haj{\l}asz in \cite{MR1458875}.
Korte--Lehrb\"ack--Tuominen proved in \cite{MR2854110}
that a pointwise $p$-Hardy inequality is equivalent to
the so-called $p$-capacity density condition of  $E$. By using 
this equivalence and the deep result on the self-improvement  of the $p$-capacity density condition,
see below, it was shown in \cite{MR2854110}
that the pointwise $p$-Hardy inequality \eqref{eq.hardy_intro} implies the integral Hardy inequality
\begin{equation}\label{e.integral_h}
\int_{X\setminus E} \frac{\lvert u(x)\rvert^p}{\dist(x,E)^p}\,d\mu(x)\le C \int_{X} g(x)^p\,d\mu(x)\,,
\end{equation}
where $E$, $u$, $g$, and $p$ are as in  \eqref{eq.hardy_intro}.

The key fact above is the self-improvement, meaning that
a  $p$-capacity density condition implies
an {\em a priori} stronger $\hat p$-capacity density condition, for some $1<\hat p<p$. 
The  $p$-capacity density condition, also known as uniform $p$-fatness, appears  in Hardy inequalities, potential theory
and PDE's, we refer to \cite{MR1207810,KLV2021}.
The self-improvement property of this condition  was first discovered by Lewis \cite{MR946438}
using potential theory in Euclidean spaces.
His result  has  been followed by other works incorporating different techniques, such as nonlinear potential theory \cite{MR1386213,MR1869615} and  
Hardy inequalities \cite{MR2564934,MR3673660}. A direct
proof for the self-improvement of pointwise Hardy inequalities by using curves is given in \cite{MR3976590,MR4116806}.
For a survey on Hardy inequalities, and their connections
to the $p$-capacity density condition, we refer to \cite{MR2723821,MR4480576,KLV2021}. 

Our main contribution is an extension of  the approach and  results in \cite{MR2854110} to the nonlocal setting of fractional Sobolev spaces and their generalizations. Some of the first steps were already taken in \cite{MR4649157}, where
the following nonlocal or fractional variant of inequality \eqref{eq.hardy_intro} was introduced.
We assume that $1< p,q<\infty$ and $0<\beta<1$.
A closed set $E$ in $X$ supports the 
pointwise $(\beta,p,q)$-Hardy inequality  if
there exists a constant $c_H>0$  such that
\begin{equation}\label{e.pw_hardy_intro}
\lvert u(x)\rvert \le c_H \dist(x,E)^{\beta} \bigl(M_{2\dist(x,E)}\bigl((G_{u,\beta,q,B(x,2\dist(x,E))})^p\bigr)(x)\bigr)^{1/p}\,,
\end{equation}
whenever $x\in X$ satisfies $0<\dist(x,E)<\diam(E)/8$ and  $u\colon X\to \R$ is a continuous function 
such that %
$u=0$ on $E$.  Here $M_{2\dist(x,E)}$ is the restricted
maximal operator as in \eqref{e.max_funct_def_restricted} and 
\[
G_{u,\beta,q,B(x,2\dist(x,E))}(y)=\biggl(\int_{B(x,2\dist(x,E))} \frac{\vert u(y)-u(z)\vert^q}{d(y,z)^{\beta q}\mu(B(y,d(y,z)))}\,d\mu(z)\biggr)^{1/q}\,,\quad y\in X\,.
\]
The following theorem is one of our main results: under suitable assumptions, the pointwise $(\beta,p,q)$-Hardy inequality implies
the  integral version of the $(\beta,p,q)$-Hardy inequality.
We also prove a closely related Theorem \ref{t.integrated}, which is  a stronger 
result  by Fatou's lemma.

\begin{theorem}\label{t.demo}
 Let $(X,d,\mu)$ be a complete geodesic space  equipped  with a doubling measure $\mu$
which satisfies  the quantitative doubling condition~\eqref{e.doubling_quant},
 with exponent $Q>0$,
and  the quantitative reverse doubling condition \eqref{e.reverse_doubling}, with exponent $\sigma>0$.
Let $1<p,q<\infty$ and $0<\beta< 1$ be such that 
\begin{equation}\label{e.restr}
Q\left(\frac{1}{p}-\frac{1}{q}\right)<\beta<\frac{\sigma}{p}\,.
\end{equation} 
Assume that an unbounded closed set $E\subset X$ supports the 
pointwise $(\beta,p,q)$-Hardy inequality. Then the integral $(\beta,p,q)$-Hardy inequality
\begin{equation}\label{eq.local_hardy_cap_intro}
\begin{split}
\int_{X\setminus E} \frac{\lvert u(x)\rvert^p}{\dist(x,E)^{\beta p}}\,d\mu(x)\le C \int_{X} \biggl(\int_{X}\frac{|u(x)-u(y)|^q}{d(x,y)^{\beta q}\mu(B(x,d(x,y)))}\,d\mu(y)\biggr)^{p/q}\,d\mu(x)
\end{split}
\end{equation}
holds whenever $u\colon X\to \R$ is a continuous function 
such that %
$u=0$ on $E$. Here the constant $C>0$ is independent of $u$.
\end{theorem}

When $q=p$, the  seminorm on the right of \eqref{eq.local_hardy_cap_intro} is a generalization of the well known Gagliardo seminorm, which is used to define a fractional Sobolev space in the standard Euclidean setting; see
\cite{MR2944369,MR4480576,MR4567945,MR2250142} for expositions on fractional Sobolev  spaces
that have   attracted plenty of interest in recent years.
When $\mu$ is the Lebesgue
measure in the Euclidean space $\R^n$, then  \eqref{e.restr} reads  
as $n(1/p-1/q)<\beta<n/p$.
This restriction on  parameters and  
additional assumptions on $X$
come from limitations of our methods, and
it is unclear to us to which extent these restrictions can be relaxed. 
On the other hand, Theorem~\ref{t.demo} and our other main results are new  even in 
a Euclidean space equipped with the Lebesgue measure.
Other capacity related sufficient conditions for the integral Hardy inequality \eqref{eq.local_hardy_cap_intro} in the special case $q=p$ are given in \cite{MR3148524,MR3277052}.
These papers employ the Lewis' result \cite{MR946438} on  the self-improvement of the Riesz $(\beta,p)$-capacity density condition in Euclidean spaces, see the discussion further below.

Our overall strategy for the proof of Theorem \ref{t.demo} 
is to show the following self-improvement of a pointwise Hardy inequality: under the asumptions of Theorem~\ref{t.demo},  the set $E$  supports a 
 pointwise $(\beta,\hat p,q)$-Hardy inequality for some $1<\hat p<p$ or, stated equivalently,
\begin{equation}\label{e.better}
\frac{\lvert u(x)\rvert^{p}}{\dist(x,E)^{\beta p}} \le C \left( M_{2\dist(x,E)}\bigl((G_{u,\beta,q,B(x,2\dist(x,E))})^{\hat p}\bigr)(x)\right)^{p/\hat p}\,,
\end{equation}
whenever $x\in X\setminus E$ and  $u$ is as in \eqref{e.pw_hardy_intro}.
Observe that \eqref{e.better} implies \eqref{eq.local_hardy_cap_intro}, since $p/\hat p>1$ and  therefore
the restricted maximal operator is bounded on  $L^{p/\hat p}(X)$.
We prove this self-improvement of a pointwise Hardy inequality, that is, Theorem \ref{t.smaller},
in two  main parts.
In the first part we prove Theorem~\ref{t.eqv_f}, which  shows that the pointwise $(\beta,p,q)$-Hardy inequality \eqref{e.pw_hardy_intro} is equivalent 
to the fact that $E$ satisfies a fractional $(\beta,p,q)$-capacity density condition, we refer to Definition \ref{d.cap_density}.
In the second part we prove Theorem~\ref{t.main_impro}, which shows
that the fractional $(\beta,p,q)$-capacity density condition is self-improving; 
this part is longer and it applies results from  \cite{LMV}.

A more specific outline for the paper  is as follows.
In Section~\ref{s.prelim} we 
introduce the standing assumptions and notation, and recall some concepts that are
needed later on. For instance, the fractional Poincar\'e inequalities 
are recalled from \cite{MR4649157}.
An important notion of a  fractional relative capacity is defined in Section \ref{s.capacitary},
where we also  give estimates for relative capacities of balls and lower
bounds for relative capacities of  closed sets in terms of Hausdorff contents. %
The fractional $(\beta,p,q)$-capacity density condition is introduced
in Section~\ref{s.fcdc}, where its equivalence with the  pointwise $(\beta,p,q)$-Hardy inequality \eqref{e.pw_hardy_intro} is also  shown. 
In Section~\ref{s.hajlasz} we define certain Haj{\l}asz--Triebel--Lizorkin seminorms and we  show in Section~\ref{s.semi} that these seminorms are often comparable to  Gagliardo-type  seminorms that are used to define the fractional capacities. 
 A  Sobolev--Poincar\'e type inequality is a key tool  here,  see Lemma \ref{lemma:Sobolev-Poincare 2}.
The  comparability  of fractional and Haj{\l}asz--Triebel--Lizorkin relative capacities follows from  the equivalence of the corresponding seminorms, as it is shown in Section \ref{s.comparison_capacities}. 
 Using the capacity comparison and  results from \cite{LMV}, we 
obtain a $q$-independence property for the fractional relative capacity in Theorem~\ref{l.q_independence}; 
we refer to \cite[Sections~3.6 and~4.4]{MR1411441} for analogous  properties of Triebel--Lizorkin capacities in Euclidean spaces. 
Using the  capacity comparison results from Section \ref{s.comparison_capacities} and \cite[Theorem 11.4]{LMV}, in Section~\ref{s.three} 
we show that the fractional $(\beta,p,q)$-capacity density condition has three ways of self-improvement, with respect to parameters $\beta$, $p$ and $q$, see Theorem \ref{t.main_impro}.
Our main results  are proved in the last Section~\ref{s.applications},
 including 
Theorem~\ref{t.integrated} that follows from
 the self-improvement of the pointwise Hardy inequality in Theorem~\ref{t.smaller}.

Finally, we make some remarks on the line  of research developed in \cite{MR946438},  \cite{MR4478471}, \cite{CILV}, \cite{LMV}.
The origins of our results can  be traced  to the seminal work by Lewis \cite{MR946438}, who  proved that a Riesz $(\beta,p)$-capacity density  condition in Euclidean spaces is doubly self-improving, that is,  self-improving  in both parameters $\beta$ and $p$. %
More recently, a Haj{\l}asz $(\beta,p)$-capacity density condition in metric  spaces was introduced in \cite{MR4478471} by using  Haj{\l}asz gradients of order $0<\beta\le 1$. It was also shown that this condition is doubly self-improving, if $X$ is a complete geodesic space. The comparability of  Riesz and Haj{\l}asz $(\beta,p)$-capacities
was shown in \cite{CILV}, where Lewis' result was also  transferred to complete geodesic  spaces.
A flexible and {\em a priori} more general Haj{\l}asz--Triebel--Lizorkin $(\beta,p,q)$-capacity density condition 
was introduced in \cite{LMV} by using fractional Haj{\l}asz gradients and Haj{\l}asz--Triebel--Lizorkin seminorms, see Definition~\ref{d.HTL_cap_density}. The Riesz $(\beta,p)$-capacity density
condition for $0<\beta<1$ is its special case, corresponding
to the indices $(\beta,p,\infty)$.
It was shown in \cite{LMV}
that the Haj{\l}asz--Triebel--Lizorkin $(\beta,p,q)$-capacity density
condition is self-improving in all three parameters $\beta$, $p$, $q$ and, in fact, 
independent of the $q$-parameter,
under suitable assumptions. Consequently, the {\em a priori} strongest Riesz $(\beta,p)$-capacity density condition
is  equivalent to a  Haj{\l}asz--Triebel--Lizorkin $(\beta,p,q)$-capacity density
condition, for all $1< q\le \infty$. Theorems \ref{l.q_independence} and \ref{t.main_impro} complement this line of research, the former proving the $q$-independence of the fractional capacity, and  the latter 
proving the self-improvement of a fractional capacity density condition in all three parameters.

\subsection*{Acknowledgements} 
The second author was supported by the Academy of Finland (project \#323960) and by the Academy of Finland via Centre of Excellence in Analysis and Dynamics Research (project \#346310).

\section{Preliminaries}\label{s.prelim}

Throughout the paper  we assume that $X=(X,d,\mu)$ is a metric measure space equipped with a metric $d$ and a 
nonnegative complete Borel
measure $\mu$ such that $0<\mu(B)<\infty$
for all balls $B\subset X$, each of which is an open set of the form \[B=B(x,r)=\{y\in X\,:\, d(y,x)<r\}\] with $x\in X$ and $r>0$.
 Under these assumptions 
the space $X$ is separable,
see \cite[Proposition~1.6]{MR2867756}.
We also assume that $\# X\ge 2$.
We let $n_0=-\infty$ if $\diam(X)=\infty$ and otherwise choose $n_0$ be the smallest integer
such that $2^{-n_0}\le 2\diam(X)$. In the latter
case, we have $\diam(X)<2^{-n_0}\le 2\diam(X)$. 

We assume throughout the paper 
that the measure $\mu$ is {\em doubling}, that is,
there is a constant $c_\mu> 1$, called
the {\em doubling constant of $\mu$}, such that
\begin{equation}\label{e.doubling}
\mu(2B) \le c_\mu\, \mu(B)
\end{equation}
for all balls $B=B(x,r)$ in $X$. 
Here and throughout the paper we use for $0<t<\infty$ the notation $tB=B(x,tr)$. 
Iteration of~\eqref{e.doubling} shows that if $\mu$ is doubling, then there exist an exponent $Q>0$ and a constant $c_Q>0$,
both  depending on $c_\mu$ only, such that the {\em quantitative doubling condition}
\begin{equation}\label{e.doubling_quant}
  \frac{\mu(B(y,r))}{\mu(B(x,R))}\ge c_Q\Bigl(\frac {r}{R}\Bigr)^Q
\end{equation}
holds whenever $y\in B(x,R)\subset X$ and $0<r<R$. %
Condition~\eqref{e.doubling_quant} holds
for $Q\ge \log_2c_\mu>0$,  but it can hold for smaller values of $Q$ as well. See~\cite[Lemma~3.3]{MR2867756}
for details.

In some of our results we also  assume that $\mu$ is
\emph{reverse doubling}, in the sense that %
there are constants  $0<\kappa<1$  and $0<c_R<1$ such that
\begin{equation}\label{e.rev_dbl_decay}
\mu(B(x,\kappa r))\le c_R\, \mu(B(x,r))
\end{equation}
for every $x\in X$ and %
$0<r<\diam(X)/2$.
If $X$ is connected and $0<\kappa<1$, then
inequality~\eqref{e.rev_dbl_decay} follows from the doubling property~\eqref{e.doubling} 
with $0<c_R=c_R(c_\mu,\kappa)<1$; see~\cite[Lemma~3.7]{MR2867756}. 
Iteration of~\eqref{e.rev_dbl_decay} shows that
if $\mu$ is reverse doubling, then there exist
an exponent $\sigma>0$ and a constant $c_\sigma>0$,
both only depending on $\kappa$ and $c_R$, such that the {\em quantitative 
 reverse doubling condition}
\begin{equation}\label{e.reverse_doubling}
 \frac{\mu(B(x,r))}{\mu(B(x,R))} \le c_\sigma\Bigl(\frac{r}{R}\Bigr)^\sigma
\end{equation}
holds for every $x\in X$ and %
$0<r<R\le 2\diam(X)$.  See the proof of \cite[Corollary~3.8]{MR2867756}.

We use the following familiar notation: \[
f_A=\vint_{A} f(y)\,d\mu(y)=\frac{1}{\mu(A)}\int_A f(y)\,d\mu(y)
\]
is the integral average of $f\in L^1(A)$ over a measurable set $A\subset X$
with $0<\mu(A)<\infty$. 
If $f:X\to \R$ is a measurable function, then the {\em centered Hardy--Littlewood
 maximal function} $Mf$ of $f$ is defined by
\begin{equation}\label{e.max_funct_def_noncentered}
Mf(x)=\sup_{r>0}\vint_{B(x,r)} \lvert f(y)\rvert\,d\mu(y)\,,\qquad x\in X\,,
\end{equation}
The sublinear operator $M$ is bounded on $L^p(X)$ 
with a constant $C(c_\mu,p)$
for every $1<p\leq \infty$, see \cite[Theorem 3.13]{MR2867756}. 
If $0<R<\infty$, we define the restricted maximal function $M_Rf$ of a measurable function $f\colon X\to \R$ by
\begin{equation}\label{e.max_funct_def_restricted}
M_Rf(x)=\sup_{0<r\le R}\vint_{B(x,r)} \lvert f(y)\rvert\,d\mu(y)\,,\qquad x\in X\,.
\end{equation}

 We write $\N=\{1,2,3,\ldots\}$. 
The characteristic function of set $A\subset X$ is denoted by $\mathbf{1}_{A}$; that is, $\mathbf{1}_{A}(x)=1$ if $x\in A$
and $\mathbf{1}_{A}(x)=0$ if $x\in X\setminus A$.
The closure of a set $A\subset X$ is denoted by $\iol{A}$. In particular, if $B\subset X$ is a ball,
then the notation $\iol{B}$ refers to the closure
of the ball $B$.

\begin{definition}
We say that a function $u\colon X\to\R$ is integrable
on balls if $u$ is $\mu$-measurable and
\[\lVert u\rVert_{L^1(B)}=\int_B \lvert u(x)\rvert\,d\mu(x)<\infty\] for all
balls $B\subset X$. In particular, for such functions the
integral average
\[
u_B = \vint_B u(x)\,d\mu(x) =\frac{1}{\mu(B)}\int_B u(x)\,d\mu(x)
\]
is well-defined whenever $B$ is a ball in $X$.
\end{definition}

We do not  always assume that the space $X$ is complete, 
and hence continuous functions are not necessarily integrable on balls.

The following definition is from \cite[Definition 2.1]{MR4649157}.
 We remark that the seminorm on the right-hand side of \eqref{e.poinc} is well-defined, even though it might be infinite in some cases. 

\begin{definition}\label{def.poinc}
Let $1\le t,p,q<\infty$ and $0<\beta<1$. 
We say that $X$  
supports a fractional $(\beta,t,p,q)$-Poincar\'e inequality if
there are constants $c_P>0$ and $\lambda \ge 1$ such that inequality
\begin{equation}\label{e.poinc}
\begin{split}
&\biggl(\vint_B \lvert u(x)-u_B\rvert^t\,d\mu(x)\biggr)^{1/t}
\\&\qquad \le c_Pr^{\beta}\biggl(\vint_{\lambda B}\biggl( \int_{\lambda B}\frac{|u(x)-u(y)|^q}{d(x,y)^{\beta q}\mu(B(x,d(x,y)))}\,d\mu(y)\biggr)^{p/q}\,d\mu(x)\biggr)^{1/p}
\end{split}
\end{equation}
holds for every ball $B=B(x_0,r)\subset X$
and for all functions $u\colon X\to \R$  that are
integrable on balls. 
In particular, the left-hand side of \eqref{e.poinc}
is finite if the right-hand side is finite.
\end{definition}

If $u\colon X\to\R$ is a measurable function, $0<\beta<1$, $1\le q<\infty$, and $A\subset X$ is a measurable set, we write 
\begin{equation}\label{e.gdef}
G_{u,\beta,q,A}(x)=\biggl(\int_{A} \frac{\vert u(x)-u(y)\vert^q}{d(x,y)^{\beta q}\mu(B(x,d(x,y)))}\,d\mu(y)\biggr)^{1/q},\qquad \text{ for every } x\in X.
\end{equation}
Observe that 
\begin{equation}\label{e.monotonicity}
G_{L(u),\beta,q,A}\le G_{u,\beta,q,A}\le G_{u,\beta,q,A'}
\end{equation}  whenever $L\colon \R\to \R$ is a $1$-Lipschitz function and $A\subset A'$.

Using this notation, 
the $(\beta,t,p,q)$-Poincar\'e inequality \eqref{e.poinc}
can be written as
\begin{equation*}%
\biggl(\vint_B \lvert u(x)-u_B\rvert^t\,d\mu(x)\biggr)^{1/t}
\le c_Pr^{\beta}\biggl(\vint_{\lambda B}G_{u,\beta,q,\lambda B}(x)^p\,d\mu(x)\biggr)^{1/p}.
\end{equation*}

The following lemma from \cite[Lemma 2.2]{MR4649157} shows that $X$ supports a $(\beta,t,p,q)$-Poincar\'e inequality
if $1\le t \le \min\{p,q\}$.  
We emphasize that  the doubling condition
on $\mu$ is the only quantitative property of $X$ that
is needed in this case. 
This result is certainly 
known among experts,  but  we include
the short proof from \cite{MR4649157} for the  convenience of the reader. 

\begin{lemma}\label{l.qp}
Assume that $1\le t,p,q<\infty$, $t\le \min\{p,q\}$, and $0<\beta<1$. Then
$X$ supports the $(\beta,t,p,q)$-Poincar\'e inequality \eqref{e.poinc}
with constants $\lambda=1$ and $c_P=c_P(\beta,t,q,c_\mu)$.
\end{lemma}

\begin{proof}
Fix a ball $B=B(x_0,r)\subset X$ and a 
function $u\colon X\to \R$ that is integrable on balls.
Then
\begin{align*}
\vint_{B} \lvert u(x)-u_B\rvert^t\,d\mu(x)
&\le \vint_{B}\vint_B \lvert u(x)-u(y)\rvert^t\,d\mu(y)\,d\mu(x)\\
&\le  \vint_{B}\biggl(\vint_B \lvert u(x)-u(y)\rvert^q\,d\mu(y)\,\biggr)^{t/q}d\mu(x)\\
&\le   \biggl(\vint_{B}\biggl(\vint_B \lvert u(x)-u(y)\rvert^q\,d\mu(y)\,\biggr)^{p/q}d\mu(x)\biggr)^{t/p}\\
&\le   r^{\beta t}\biggl(\vint_{B}\biggl(\int_B \frac{\lvert u(x)-u(y)\rvert^q}{r^{\beta q}\mu(B)}\,d\mu(y)\,\biggr)^{p/q}d\mu(x)\biggr)^{t/p}\\
&\le   Cr^{\beta t}\biggl(\vint_{B}\biggl(\int_B \frac{\lvert u(x)-u(y)\rvert^q}{d(x,y)^{\beta q}\mu(4B)}\,d\mu(y)\,\biggr)^{p/q}d\mu(x)\biggr)^{t/p}\\
&\le   Cr^{\beta t}\biggl(\vint_{B}\biggl(\int_B \frac{\lvert u(x)-u(y)\rvert^q}{d(x,y)^{\beta q}\, \mu(B(x,d(x,y)))}\,d\mu(y)\,\biggr)^{p/q}d\mu(x)\biggr)^{t/p}.
\end{align*}
This yields the desired inequality
\eqref{e.poinc} with $\lambda=1$ and
$c_P=c_P(\beta,t,q,c_\mu)$.
\end{proof}

For the proof of Lemma \ref{l.qp_stronger} we use the following variant of the Kolmogorov lemma, which is essentially
\cite[Lemma 4.22]{MR1800917}. See also \cite[p.~485]{MR807149}.

\begin{lemma}\label{l.kolmogorov}
Assume that  $1\le p<p^*<\infty$,
$B\subset X$ is a ball and $u\colon B\to \R$
is a measurable function. Assume that there exists a constant
$C_0>0$ such that
\[
\mu(\{x\in B\,:\, \lvert u(x)\rvert > s\})\,s^{p^*}\le C_0 
\]
for each $s>0$. Then
\[
\bigg(\vint_B \lvert u(x)\rvert^p\,d\mu(x)\bigg)^{1/p}\le 2^{1/p}\bigg(\frac{C_0 p}{p^*-p}\bigg)^{1/p^*} \mu(B)^{-1/p^*}\,.
\]
\end{lemma}

If
$p\le q$, then $(\beta,p,p,q)$-Poincar\'e inequalities follow from Lemma \ref{l.qp}.
The following lemma shows that these inequalities hold 
also in the case $p> q$,  at least under an additional mild geometric assumption. To establish this useful result, we  adapt the proof of \cite[Theorem 3.4]{MR4649157}, see also \cite[pp.\ 95--97]{MR2867756}. The outline of this argument
is originally from \cite{MR1336257}.

\begin{lemma}\label{l.qp_stronger}
Assume that $\mu$ satisfies the quantitative reverse doubling condition \eqref{e.reverse_doubling}, with
constants $\sigma>0$ and $c_\sigma>0$.
Assume that $1\le p,q<\infty$ and $0<\beta<1$. Then
$X$ supports the $(\beta,p,p,q)$-Poincar\'e inequality \eqref{e.poinc}
with constants $\lambda=2$ and $c_P=c_P(\beta,p,q,c_\mu,\sigma,c_\sigma)$.
\end{lemma}

\begin{proof}
Since $\mu$ is doubling, also the quantitative doubling condition \eqref{e.doubling_quant} holds, with exponent $Q>0$ and constant $c_Q>0$, both depending on $c_\mu$ only. By
increasing $Q$ if necessary, we may assume that
$Q=Q(c_\mu,\beta,p)>\beta p$. We denote  $p^*=Qp/(Q-\beta p)>p$.

Let $B=B(x_0,r)$ be a ball  and let $u\in L^\infty(2B)$.
By adapting the proof of \cite[Theorem 3.4]{MR4649157}, we obtain a constant $C=C(p,\beta,\sigma,c_\mu,c_\sigma)$ such that
\begin{equation}\label{e.mazya_goal}
\begin{split}
\mu(\{x\in B :  \lvert u(x)-u_{2B}\rvert > s\})\, s^{p^*}
\le Cr^{\beta p^*}\mu(B)^{1-p^*/p}\biggl( \int_{2B} G_{u,\beta,q,2B}(y)^p
\,d\mu(y)\biggr)^{\frac{p^*}{p}}
\end{split}
\end{equation}
for all $s>0$.
More specifically,  it is clear from the proof that the $(\beta,1,p,q)$-Poincar\'e inequality
given by Lemma \ref{l.qp} suffices for our adaptation; in particular, this inequality is used to replace both
$(\beta,p,p,p)$-Poincar\'e and $(\beta,1,p,p)$-Poincar\'e inequalities
in the proof of \cite[Theorem 3.4]{MR4649157}.
 Lemma \ref{l.kolmogorov} then implies
\begin{equation}\label{e.general}
\begin{split}
\bigg(\vint_B \lvert u(x)-u_B\rvert^p\,d\mu(x)\bigg)^{1/p}
&\le 
2\bigg(\vint_B \lvert u(x)-u_{2B}\rvert^p\,d\mu(x)\bigg)^{1/p}\\
&\le Cr^{\beta}\mu(B)^{1/p^*-1/p}\biggl( \int_{2B} G_{u,\beta,q,2B}(y)^p\,d\mu(y)\biggr)^{\frac{1}{p}}\mu(B)^{-1/p^*}\\
&\le  Cr^{\beta}\biggl( \vint_{2B} G_{u,\beta,q,2B}(y)^p\,d\mu(y)\biggr)^{\frac{1}{p}}\,. 
\end{split}
\end{equation}
If $u\colon X\to \R$ is merely integrable on balls, then the  $(\beta,p,p,q)$-Poincar\'e inequality
follows by applying \eqref{e.general}
to the functions $u_k=\max\{-k,\min\{u,k\}\}\in L^\infty(2B)$, $k\in\N$, and then using Fatou's lemma. Indeed, observe 
that $\liminf_{k\to\infty} \lvert u_k(x)-(u_k)_{B}\rvert^p=\lvert u(x)-u_{B}\rvert^p$ 
for all $x\in X$ and that $G_{u_k,\beta,q,2B}\le G_{u,\beta,q,2B}$ for all
$k\in\N$.
\end{proof}

\section{Fractional relative capacity}\label{s.capacitary}

Here  we define a variant  of the 
fractional relative capacity, 
compare to~\cite[Definition~7.1]{MR3605166} and see  also \cite{MR3331699} and \cite[\S 11]{Mazya2011}.
In Lemma~\ref{l.frac_cap_balls} we give estimates for relative capacities of balls.
Lower bounds for relative capacities of closed sets in terms of a 
Hausdroff content are given in Lemma~\ref{l.codim}. At the end of this section we recall a fractional version of
Maz'ya's capacitary Poincar\'e inequality.

\begin{definition}\label{d.frac_capacity}
Let $0<\beta<1$, $1\le q,p<\infty$ and $\Lambda\ge 2$. 
Let $B\subset X$ be a ball and let $E\subset \iol{B}$ be a closed set. Then we write
\[
\cp_{\beta,p,q} (E,2B,\Lambda B) = \inf_\varphi \int_{\Lambda B} \biggl( \int_{\Lambda B}\frac{|\varphi(x)-\varphi(y)|^q}{d(x,y)^{\beta q}\mu(B(x,d(x,y)))}\,d\mu(y)\biggr)^{p/q}\,d\mu(x)\,,
\]
where the infimum is taken over all continuous functions 
$\varphi\colon X\to \R$  such that $\varphi(x)\ge 1$ for every $x\in E$ and 
$\varphi(x)=0$ for every $x\in X\setminus 2B$.
\end{definition}

 \begin{remark}\label{r.fremark}
Suppose that $\varphi\colon X\to \R$  is a continuous function such that $\varphi(x)\ge 1$ for every $x\in E$ and 
$\varphi(x)=0$ for every $x\in X\setminus 2B$, where $B$ and $E$ are
as in Definition \ref{d.frac_capacity}.
By considering the function $v=\max\{0,\min\{\varphi,1\}\}$
and using inequality~\eqref{e.monotonicity}, it is clear
that the infimum in Definition \ref{d.frac_capacity} can be taken
either over {\em all continuous and bounded functions} or 
over {\em all continuous functions that are integrable on balls},  in both cases such that $\varphi(x)\ge 1$ for every $x\in E$ and 
$\varphi(x)=0$ for every $x\in X\setminus 2B$. 
In particular, our Definition \ref{d.frac_capacity} coincides with the relative fractional capacity defined in \cite[Definition 4.1]{MR4649157}.
\end{remark} 

Let us  estimate relative capacities of balls. 
Lemma \ref{l.frac_cap_balls}(i) and a monotonicity argument show that the
fractional relative capacities in Definition~\ref{d.frac_capacity} are always finite. 
In order to prove this, we use the 
following lemma from \cite[Lemma 4.2]{MR4649157}.

\begin{lemma}\label{l.alpha}
Let $\alpha>0$.
There is a constant $C(\alpha,c_\mu)>0$ such that
\[
\int_{B(x,r)} \frac{d(x,y)^{\alpha}}{\mu(B(x,d(x,y)))} \, d\mu(y)\leq C(\alpha,c_\mu)   r^{\alpha}\
\]
for every $x\in X$ and $r>0$.
\end{lemma}

\begin{lemma}\label{l.frac_cap_balls}
Let $1\le p,q<\infty$ and  $0<\beta <1$.
\begin{itemize}
\item[(i)]
 Assume that $\Lambda \ge 2$, 
$x_0\in X$ and $r>0$. Then
\[
\mathrm{cap}_{\beta,p,q} (\overline{B(x_0,r)},B(x_0,2r),B(x_0,\Lambda r))\le  C(c_\mu,\beta,p,q,\Lambda) r^{-\beta p}\mu(B(x_0,r))\,.
\]
\item[(ii)]
 Assume that $X$ is connected
and that  $\Lambda> 2$.
Assume that $x_0\in X$ and $r>0$ is such that $r<(1/8)\diam(X)$.
Then
\[
r^{-\beta p}\mu(B(x_0,r))\le C(\beta,p,q, c_\mu,\Lambda)\mathrm{cap}_{\beta,p,q} (\overline{B(x_0,r)},B(x_0,2r),B(x_0,\Lambda r))\,.
\]
\end{itemize}
\end{lemma}

\begin{proof}
We show (i) first. Denote $B=B(x_0,r)$. Let
\[
\varphi(x)=\max \Bigl\{0,1-r^{-1}\dist(x,B)\Bigr\}\,
\]
for every $x\in X$. 
Then $0\le\varphi\le 1$ in $X$, $\varphi=1$ in $\iol{B}$, $\varphi=0$ in $X\setminus 2B$, 
and  $\varphi$ is an $r^{-1}$-Lipschitz function in $X$.
Since $\varphi$ is admissible for the relative capacity, we have
\begin{align*}
\mathrm{cap}_{\beta,p,q} (\overline{B},2B,\Lambda B) &\le \int_{\Lambda B} \biggl( \int_{\Lambda B}\frac{|\varphi(x)-\varphi(y)|^q}{d(x,y)^{\beta q}\mu(B(x,d(x,y)))}\,d\mu(y)\biggr)^{p/q}\,d\mu(x)\,.
\end{align*}
  Fix $x\in \Lambda B$. Since $|\varphi(x)-\varphi(y)| \leq  d(x,y)/r$
  for each $y\in \Lambda B$, by Lemma \ref{l.alpha} we have
  \[
    \int_{\Lambda B} \frac{|\varphi(x)-\varphi(y)|^q}{d(x,y)^{\beta q}\mu(B(x,d(x,y)))}\,d\mu(y) \leq
   r^{-q} \int_{B(x,2\Lambda r)} \frac{d(x,y)^{q(1-\beta)}}{\mu(B(x,d(x,y)))}\,d\mu(y) \leq C(\beta,q,c_\mu,\Lambda) r^{-\beta q}.
  \]
Hence, we have
\[
\mathrm{cap}_{\beta,p,q} (\overline{B},2B,\Lambda B) \le C(\beta,p,q,c_\mu,\Lambda) r^{-\beta p}\mu(\Lambda B)\le C(\beta,p,q,c_\mu,\Lambda) r^{-\beta p}\mu(B)\,,
\]
and condition (i) follows.

Next we prove (ii). Denote by $B=B(x_0,r)$ and  $\lambda=\min\{3,\Lambda\}>2$.
Choose any $\varphi$ from the non-empty set of test functions for
the relative capacity $\cp_{\beta,p,q}(\overline{B},2B,\Lambda B)$.
By replacing $\varphi$ with $\max\{0,\min\{\varphi,1\}\}$, if necessary, we may assume that $0\le \varphi \le 1$ in $X$. 
Thus $\varphi$ is continuous on $X$,  $\varphi=1$ on $\overline{B}$, and $\varphi=0$ on $B\setminus 2B$. 
Since $X$ is connected, there exists a constant $0<c_R=c(c_\mu,\Lambda)<1$ such that 
inequality~\eqref{e.rev_dbl_decay} holds with $\kappa=2/\lambda$.
In particular, since $\lambda r<4r<\diam(X)/2$, we have
 \[
\mu(2B)=\mu(B(x_0,2r))\le c_R\,\mu(B(x_0,\lambda r))\le c_R\,\mu(B(x_0,\Lambda r))= c_R\,\mu(\Lambda B)\,.
\]
Therefore
\begin{align*}
0\le \varphi_{\Lambda B}=\vint_{\Lambda B} \varphi(y)\,d\mu(y) \le \frac{\mu(2B)}{\mu(\Lambda B)}
\le c_R < 1.
\end{align*}
As a consequence, since $\varphi=1$ in $B$, we find that
\begin{align*}
1&=  \frac{1}{1-c_R}\vint_B (1-c_R)\,d\mu(x)\\&\le 
 \frac{1}{1-c_R}\vint_B \lvert \varphi(x)- \varphi_{\Lambda B}\rvert\,d\mu(x)
\le C(c_\mu,\Lambda) \vint_{\Lambda B} \lvert \varphi(x)- \varphi_{\Lambda B}\rvert\,d\mu(x)\,.
\end{align*}
Applying the $(\beta,1,p,q)$-Poincar\'e inequality, see Lemma \ref{l.qp}, we get
\[
1\le C(\beta,p,q,c_\mu,\Lambda) r^{\beta p}\vint_{\Lambda B}\biggl( \int_{\Lambda B}\frac{|\varphi(x)-\varphi(y)|^q}{d(x,y)^{\beta q}\mu(B(x,d(x,y)))}\,d\mu(y)\biggr)^{p/q}\,d\mu(x)\,.
\]
Multiplying both sides
by $r^{-\beta p}\mu(\Lambda B)$ and using
inequality $\mu(B)\le \mu(\Lambda B)$, we get
\[
r^{-\beta p}\mu(B)\le C(\beta,p,q,c_\mu,\Lambda) \int_{\Lambda B}\biggl( \int_{\Lambda B}\frac{|\varphi(x)-\varphi(y)|^q}{d(x,y)^{\beta q}\mu(B(x,d(x,y)))}\,d\mu(y)\biggr)^{p/q}\,d\mu(x)\,.
\]
The claim (ii) follows by taking infimum over all test functions $\varphi$ as above.
\end{proof}

\begin{example}\label{e.counter}
Consider the Euclidean space $X=\R^n$ equipped with the
standard Lebesgue measure $\mu$ and the Euclidean distance $d$.
In this setting we have, by computations in \cite[\S 2]{MR2085428}, 
\[\mathrm{cap}_{\beta,p,q} (K,B(x_0,2r),B(x_0,2r))=0\] 
whenever $K\subset B(x_0,2r)$ is a compact set, and $0<\beta<1$ and $1\le p<\infty$ are such that $\beta p<1$.
This shows that the  part (ii) of Lemma \ref{l.frac_cap_balls} does not hold in general with $\Lambda=2$.
\end{example}

We also need capacitary  lower bounds for more general sets than balls,
as in part (ii) of Lemma \ref{l.frac_cap_balls}.
For this purpose, we use  a suitable Hausdorff content.

\begin{definition}\label{d.hcc}
Let $0<\rho\le\infty$ and $d\ge 0$. The $\rho$-restricted Hausdorff content of codimension $d\ge 0$ 
of a set $F\subset X$ is defined by  
\[
\Ha^{\mu,d}_\rho(F)=\inf\Biggl\{\sum_{k} \mu(B(x_k,r_k))\,r_k^{-d} :
F\subset\bigcup_{k} B(x_k,r_k)\text{ and } 0<r_k\leq \rho  \Biggr\}.
\]
\end{definition}

The following lemma is from \cite[Lemma 4.6]{MR4649157}.

\begin{lemma}\label{l.codim}
Let $0<\beta<1$,  $1\le p,q<\infty$, $0\le \eta<p$ and $\Lambda> 2$.
Assume that $\mu$ is 
reverse doubling, with constants $\kappa=2/\Lambda$
and $0<c_R<1$ in~\eqref{e.rev_dbl_decay}.
Let $B=B(x_0,r)\subset X$ be a ball with $r<\diam(X)/(2\Lambda)$,
and assume that $E\subset \iol{B}$ is a closed set.
Then
\[
\mathcal{H}^{\mu,\beta\eta}_{5\Lambda r}(E) \le 
C(\beta,q,p,\eta,c_R,c_\mu,\Lambda) r^{\beta(p-\eta)}\cp_{\beta,p,q}(E,2B,\Lambda B).
\]
\end{lemma}

We also need the following fractional version of Maz$'$ya's capacitary Poincar\'e inequality from \cite[Lemma 4.3]{MR4649157}.

\begin{lemma}\label{t.Mazya}
Let $1\le p\le t<\infty$, $1\leq q < \infty$, $0<\beta<1$,  and $\Lambda\ge 2$. 
Assume that $X$ supports a $(\beta,t,p,q)$-Poincar\'e inequality
with constants $c_P>0$ and $\lambda\ge 1$.
Let $u\colon X\to \R $ be a continuous function and let
\[
Z=\{x\in X : u(x)=0 \}.
\] 
Then, for all balls $B=B(x_0,r)\subset X$,
\begin{equation}\label{e.foM}\begin{split}
\biggl( \intav_{\Lambda B} |u(x)|^t\, d\mu(x)  \biggr)^{p/t}
\le \frac{C(\beta,q,p,c_\mu,c_P,\Lambda)}{\cp_{\beta,p,q}(\iol{B}\cap Z, 2B,\Lambda B)}
  \int_{\lambda\Lambda B} G_{u,\beta,q,\lambda\Lambda B}(x)^p \,d\mu(x).
  \end{split}
\end{equation}
\end{lemma}

\section{Fractional capacity density condition}\label{s.fcdc}

 In this section, we first define the fractional capacity density condition, see Definition \ref{d.cap_density}. We also define the boundary Poincar\'e and pointwise Hardy inequalities in Definition \ref{d.bp_ph}.
 The counterparts
of Definitions \ref{d.cap_density} and \ref{d.bp_ph}
in the local case of upper gradients can be found in \cite{MR2854110}.
We prove three Lemmata \ref{l.bdry_poincare}, \ref{l.pointwise} and \ref{l.phtocd}, whose combination 
gives rise to a characterization of 
the fractional capacity density condition in terms of 
a boundary Poincar\'e  and a pointwise
Hardy inequality, we refer to  
Theorem \ref{t.eqv_f}. 

Motivated by  \cite{MR4649157} and \cite{LMV}, we define the fractional capacity density condition as follows.

\begin{definition}\label{d.cap_density}
Let $1\le p,q<\infty$, $0<\beta<1$, and let $E\subset X$ be a closed set.
We say that $E$ satisfies the fractional $(\beta,p,q)$-capacity density condition if there
are constants  $c_0>0$ and $\Lambda>2$ such that
\begin{equation}\label{e.fractionalk_capacity_density_condition}
\begin{split}
&\cp_{\beta,p,q}(E\cap \overline{B(x,r)},B(x,2r),B(x,\Lambda r))\ge 
 c_0 \, \cp_{\beta,p,q}(\overline{B(x,r)},B(x,2r),B(x,\Lambda r)) %
 \end{split}
\end{equation}
for all $x\in E$ and all $0<r<(1/8)\diam(E)$.
\end{definition}

In the light of Example \ref{e.counter}, we exclude the  choice $\Lambda=2$ in Definition \ref{d.cap_density}.
Boundary Poincar\'e inequalities and pointwise Hardy inequalities are defined as follows, as in \cite{MR4649157}.

\begin{definition}\label{d.bp_ph}
Let $1\le t,p,q<\infty$, $0<\beta<1$, and let $E\subset X$ be a closed set.
\begin{itemize}
\item[(i)] We say that  $E$ supports the boundary $(\beta,t,p,q)$-Poincar\'e  
inequality if there are constants $c_b>0$ and $\lambda\ge 1$ such that
\begin{equation}\label{eq.bdry_poinc_cap}
\begin{split}
&\left(\vint_{B} \lvert u(x)\rvert^t\,d\mu(x)\right)^{1/t}
\\&\qquad \le c_b R^{\beta}\left(\vint_{\lambda B} \biggl(\int_{\lambda B}\frac{|u(x)-u(y)|^q}{d(x,y)^{\beta q}\mu(B(x,d(x,y)))}\,d\mu(y)\biggr)^{p/q}\,d\mu(x)\right)^{1/p}\,,
\end{split}
\end{equation}
whenever $u\colon X\to \R$ is a continuous function 
such that %
$u=0$ on $E$ and $B=B(x_0,R)$ is a ball with
$x_0\in E$ and $0<R<\diam(E)/8$.
\item[(ii)] We say that  $E$ supports the 
pointwise $(\beta,p,q)$-Hardy inequality if
there exists a constant $c_H>0$  such that
\begin{equation}\label{e.pw_hardy}
\lvert u(x)\rvert \le c_H \dist(x,E)^{\beta} \bigl(M_{2\dist(x,E)}\bigl((G_{u,\beta,q,B(x,2\dist(x,E))})^p\bigr)(x)\bigr)^{1/p}\,,
\end{equation}
whenever $x\in X$ satisfies $0<\dist(x,E)<\diam(E)/8$ and  $u\colon X\to \R$ is a continuous function 
such that %
$u=0$ on $E$. Here $M_{2\dist(x,E)}$ is as in \eqref{e.max_funct_def_restricted} and $G_{u,\beta,q,B(x,2\dist(x,E))}$ is as in \eqref{e.gdef}.
\end{itemize}
\end{definition}

The following result is an adaptation of \cite[Theorem 4.7]{MR4649157}, 
where a dimension bound on $E$ is assumed. Here we assume instead that $E$ satisfies
the fractional capacity density condition.

\begin{lemma}\label{l.bdry_poincare}
Let $1\le p\le t<\infty$, $1\le q<\infty$ and $0<\beta<1$. 
Assume that $X$ is connected 
and supports a fractional $(\beta,t,p,q)$-Poincar\'e inequality,
with constants $c_P>0$ and $\lambda\ge 1$.
Let
$E\subset X$ be a closed set that satisfies the fractional $(\beta,p,q)$-capacity density condition,
with constants $c_0>0$ and $\Lambda>2$.
Then $E$ supports the boundary $(\beta,t,p,q)$-Poincar\'e  
inequality, with constants $c_b=C(\beta,q,p,c_\mu,c_P,c_0,\lambda,\Lambda)$ and $\lambda$.
\end{lemma} 

\begin{proof}
Let $u\colon X\to \R$ be a continuous function 
such that %
$u=0$ on $E$, and let $B(x_0,R)$ be a ball with
$x_0\in E$ and $0<R<\diam(E)/8$.
We prove the claim \eqref{eq.bdry_poinc_cap} the ball $B(x_0,R)$, but for convenience
we write during the proof that $B=B(x_0,r)=B(x_0,R/\Lambda)$. 
Observe that $r=R/\Lambda<R<\diam(E)/8$.
Hence, by the assumed  capacity density condition
and part (ii) of Lemma \ref{l.frac_cap_balls},
\begin{equation}\label{e.confusing}
\begin{split}
&\cp_{\beta,p,q}(E\cap \overline{B},2B,\Lambda B)\ge 
 c_0 \: \cp_{\beta,p,q}(\overline{B},2 B,\Lambda B)
\ge  C(\beta,p,q, c_\mu,c_0,\Lambda)r^{-\beta p}\mu(B)\,.
\end{split}
 \end{equation}
Write $Z=\{x\in X :  u(x)=0\}\supset E$. By monotonicity of the capacity, inequality \eqref{e.confusing} and the doubling condition of the measure $\mu$, we have
\[
\frac{1}{\cp_{\beta,p,q}(\iol{B}\cap Z,2B,\Lambda B)}
\le \frac{1}{\cp_{\beta,p,q}(\iol{B}\cap E,2B,\Lambda B)}
\le \frac{C r^{\beta p}}{\mu(B)}\le \frac{C R^{\beta p}}{\mu(\lambda \Lambda B)}.
\]
The desired inequality \eqref{eq.bdry_poinc_cap}, for  the ball $B(x_0,R)=B(x_0,\Lambda r)$, follows from Lemma~\ref{t.Mazya}.
\end{proof}

Next we show that if $E$ supports
the boundary Poincar\'e inequality, then $E$ supports  a pointwise Hardy inequality.  The proof is inspired by \cite[Theorem 5.1]{MR4649157},
where a dimension bound on $E$ is assumed instead of the boundary Poincar\'e inequality.

\begin{lemma}\label{l.pointwise}
Let $1\le t,p,q<\infty$ and $0<\beta<1$. 
Assume that a closed set
$E\subset X$ supports the boundary $(\beta,t,p,q)$-Poincar\'e inequality \eqref{eq.bdry_poinc_cap},
with constants $c_b>0$ and $\lambda\ge 1$.
Then $E$ supports the pointwise $(\beta,p,q)$-Hardy  
inequality \eqref{e.pw_hardy}, with $c_H=C(\beta,q,c_\mu,c_b,\lambda)$.
\end{lemma}

\begin{proof}
Let $u\colon X\to \R$ be a continuous function 
such that %
$u=0$ on $E$.
By replacing $u$ with functions $u_k=\min\{k,\lvert u\rvert\}$, where $k\in\N$,  
and approximating,
we may assume that $u$ is integrable on balls.

Fix $x\in X$ with $0<\dist(x,E)<\diam(E)/8$
and let $B=B(x,2\dist(x,E))$.
Write $r=2\dist(x,E)>0$ and fix $w\in E$ such that $d(x,w)<(3/2)\dist(x,E)$. Then
\[
\widetilde B=B(w,r/(4\lambda))\subset B,
\]
where $\lambda\ge 1$ is the constant in the boundary $(\beta,t,p,q)$-Poincar\'e inequality \eqref{eq.bdry_poinc_cap}, 
and 
\begin{equation}\label{eq.to_three_cap}
\begin{split}
\lvert u(x)\rvert = \lvert u(x)-u_B+u_B - u_{\widetilde B}+u_{\widetilde B}\rvert%
\le \lvert u(x)-u_B\rvert+\lvert u_B-u_{\widetilde B}\rvert + \lvert u_{\widetilde B}\rvert.
\end{split} 
\end{equation}
We estimate each of the terms on the right-hand side separately.

First observe  that $\lambda \widetilde B\subset B$
and that the measures of these two balls are comparable, 
with constants only depending on $c_\mu$. %
Hence, by applying inequality \eqref{eq.bdry_poinc_cap} for the ball $\widetilde B$,
whose center is $w\in E$ and radius is $r/(4\lambda)<\diam(E)/8$, we obtain 
\begin{align*}
\lvert u_{\widetilde B}\rvert&\le
\vint_{\widetilde B} \lvert u(y)\rvert\,d\mu(y) 
\le \left(\vint_{\widetilde B} \lvert u(y)\rvert^t\,d\mu(y) \right)^{1/t}
\le c_b \,r^{\beta}\biggl(\vint_{\lambda \widetilde B} \mathbf{1}_{B}(y)
 G_{u,\beta,q,\lambda \widetilde B}(y)^p\,d\mu(y)\biggr)^{1/p}\\
&\le C(c_b,c_\mu) \,r^{\beta}\biggl(\vint_{B} G_{u,\beta,q,B}(y)^p\,  d\mu(y)\biggr)^{1/p}\\&\le C(c_b,c_\mu,\beta) \dist(x,E)^{\beta} \bigl(M_{2\dist(x,E)}\bigl((G_{u,\beta,q,B})^p\bigr)(x)\bigr)^{1/p}. 
\end{align*}
Recall from Lemma \ref{l.qp} that $X$ supports a $(\beta,1,p,q)$-Poincar\'e inequality, with constants $\lambda=1$ and $C(\beta,q,c_\mu)$.
By the doubling condition, followed by the $(\beta,1,p,q)$-Poincar\'e inequality, we obtain 
\begin{align*}
\lvert u_B-u_{\widetilde B}\rvert&\le C(c_\mu,\lambda) \vint_{B} \lvert u(y)-u_B\rvert\,d\mu(y)\\&
\le  C(\beta,q,c_\mu,\lambda)r^{\beta}\biggl(\vint_{B}  G_{u,\beta,q,B}(y)^p \,d\mu(y)\biggr)^{1/p}
\\&\le C(\beta,q,c_\mu,\lambda)\dist(x,E)^{\beta}  \bigl(M_{2\dist(x,E)}\bigl((G_{u,\beta,q,B})^p\bigr)(x)\bigr)^{1/p}. 
\end{align*}

 In order to estimate the term $\lvert u(x)-u_B\rvert$, we write
$B_j=2^{-j}B=B(x,2^{-j}r)$ for $j=0,1,2,\ldots$.
Since   $\lim_{j\to \infty} u_{B_j}= u(x)$, we find that
\begin{align*}
\lvert u(x)-u_B\rvert&\le \sum_{j=0}^\infty \lvert u_{B_{j+1}}-u_{B_j}\rvert%
 \le c_\mu\sum_{j=0}^\infty  \vint_{B_j} \lvert u(y)-u_{B_j}\rvert\,d\mu(y)\\
&\le C(\beta,q,c_\mu) \sum_{j=0}^\infty (2^{-j}r)^{\beta}\biggl(\vint_{{B_j}}
G_{u,\beta,q,B_j}(y)^p\,d\mu(y)\biggr)^{1/p}\\
&\le C(\beta,q,c_\mu) r^{\beta}\sum_{j=0}^\infty 2^{-j\beta}\biggl(\vint_{{B_j}}
G_{u,\beta,q,B}(y)^p \,d\mu(y)\biggr)^{1/p}\\
&\le C(\beta,q,c_\mu) r^{\beta}\sum_{j=0}^\infty 2^{-j\beta} \bigl(M_{2\dist(x,E)}\bigl((G_{u,\beta,q,B})^p\bigr)(x)\bigr)^{1/p}\\
&=C(\beta,q,c_\mu) \dist(x,E)^{\beta}\bigl(M_{2\dist(x,E)}\bigl((G_{u,\beta,q,B})^p\bigr)(x)\bigr)^{1/p}.
\end{align*}
The claim follows from~\eqref{eq.to_three_cap} and the estimates above. 
\end{proof}

The following lemma  closes the circuit by showing that the pointwise Hardy inequality
implies the fractional capacity density condition.
The proof given below is adapted from the proof of  \cite[Theorem 6.23]{KLV2021} which, in turn, is based on the proof of \cite[Lemma 2]{MR2854110}.

\begin{lemma}\label{l.phtocd}
Let $1\le p,q<\infty$, $0<\beta<1$ and 
$\Lambda> 2$.
Assume that $X$ is connected. Let 
$E\subset X$ be a closed set that supports the pointwise $(\beta,p,q)$-Hardy  
inequality~\eqref{e.pw_hardy}, with a constant $c_H>0$. Then $E$  satisfies the fractional $(\beta,p,q)$-capacity density condition,
with constants $c_0=C(\beta,p,q,c_\mu,c_H,\Lambda)>0$ and $\Lambda$.
\end{lemma}

\begin{proof}
Let $x_0\in E$ and $0<R<(1/8)\diam(E)$.
Part (i) of Lemma \ref{l.frac_cap_balls}  implies that
\[
\cp_{\beta,p,q}(\overline{B(x_0,R)},B(x_0,2R),B(x_0,\Lambda R))\le 
C(c_\mu,\beta,p,q,\Lambda)R^{-\beta p}\mu(B(x_0,R))
\]
By Definitions \ref{d.cap_density} and \ref{d.frac_capacity}, it suffices to find a constant  $C=C(\beta,p,q,c_\mu,c_H,\Lambda)>0$ such that
\begin{equation}\label{e.suffhere}
R^{-\beta p}\mu(B(x_0,R))\le C\int_{B(x_0,\Lambda R)} \biggl( \int_{B(x_0,\Lambda R)}\frac{|v(x)-v(y)|^q}{d(x,y)^{\beta q}\mu(B(x,d(x,y)))}\,d\mu(y)\biggr)^{p/q}\,d\mu(x)
\end{equation}
for all continuous functions
$v\colon X\to \R$ such that $v(x)\ge 1$ for every $x\in E\cap\overline{B(x_0,R)}$ and 
$v(x)=0$ for every $x\in X\setminus B(x_0,2R)$.
Moreover, by replacing $v$ with $\max\{0,\min\{v,1\}\}$ if necessary, we may assume that $0\leq v\leq 1$ and that $v$ is integrable on balls.

Let $Q>0$ and $c_Q>0$ be the constants in the quantitative
doubling condition \eqref{e.doubling_quant},
both depending on $c_\mu$ only.
Let $\gamma=\frac16$. First assume that
\[
\vint_{B(x_0,R)} v(x)\,d\mu(x)> \frac{c_Q\gamma^Q}4.
\] 
Let $\lambda=\min\{3,\Lambda\}>2$.
 Since $X$ is connected,  there exists a constant $0<c_R=c(c_\mu,\Lambda)<1$ such that 
inequality~\eqref{e.rev_dbl_decay} holds with $\kappa=2/\lambda$.
In particular, since $\lambda R<4R<\diam(X)/2$, we have
 \[
\mu(B(x_0,2R))\le c_R\,\mu(B(x_0,\lambda R))\le c_R\,\mu(B(x_0,\Lambda R))\,.
\]
Hence, we see that
\begin{equation}\label{e.mest}
\mu(B(x_0,\Lambda R)\setminus B(x_0,2R))=\mu(B(x_0,\Lambda R))-\mu(B(x_0,2R))\ge
(1-c_R)\mu(B(x_0,\Lambda R)).
\end{equation}
Observe that  $v_{B(x_0,\Lambda R)\setminus B(x_0,2R)}=0$.
By inequality \eqref{e.mest} and Lemma \ref{l.qp},
\begin{align*}
1 &< \frac{4}{c_Q\gamma^{Q}} \vint_{B(x_0,R)} v(x)\,d\mu(x) 
\le C(c_\mu,\Lambda) \vint_{B(x_0,\Lambda R)} \lvert v(x)-v_{B(x_0,\Lambda R)\setminus B(x_0,2R)}\rvert\,d\mu(x)\\
&\le C(c_\mu,\Lambda)\vint_{B(x_0,\Lambda R)} \lvert v(x)-v_{B(x_0,\Lambda R)}\rvert\,d\mu(x)+C(c_\mu,\Lambda)\lvert v_{B(x_0,\Lambda R)}-v_{B(x_0,\Lambda R)\setminus B(x_0,2R)}\rvert\\
&\le C(c_\mu,\Lambda)\vint_{B(x_0,\Lambda R)} \lvert v(x)-v_{B(x_0,\Lambda R)}\rvert\,d\mu(x)\\
&\le C(\beta,q,c_\mu,\Lambda) R^\beta \biggl(\vint_{B(x_0,\Lambda R)} G_{v,\beta,q,B(x_0,\Lambda R)}(x)^p\,d\mu(x)\biggr)^{\frac1p},
\end{align*}
from which inequality~\eqref{e.suffhere} easily follows.

Next we assume that 
\[
\vint_{B(x_0,R)} v(x)\,d\mu(x)\le\frac{c_Q\gamma^Q}4.
\] 
Let $F=\{x\in  B(x_0,\gamma R) : v(x) < \frac12 \}$.  
Since $v=1$ in $E\cap\overline{B(x_0,R)}$, we have $F\subset X\setminus E$.
By definition $v\ge\frac12$ in $B(x_0,\gamma R)\setminus F$,
and thus
\[
\mu(B(x_0,\gamma R)\setminus F) 
\le 2\int_{B(x_0,\gamma R)\setminus F} v(x)\,d\mu(x) 
\le 2\int_{B(x_0,R)} v(x)\,d\mu(x)
\le \frac{c_Q\gamma^Q} 2 \mu(B(x_0,R)).
\]
This together with the quantitative doubling condition \eqref{e.doubling_quant} implies
\begin{equation}\label{e.mitta_F}
\begin{split}
\mu(F)
&=\mu(B(x_0,\gamma R))-\mu(B(x_0,\gamma R)\setminus F)\\
&\ge c_Q\gamma^Q \mu(B(x_0,R))-\frac{c_Q\gamma^Q} 2  \mu(B(x_0,R))
=\frac{c_Q\gamma^Q} 2 \mu(B(x_0,R)).
\end{split}
\end{equation}
Let 
\[
\psi(x)=\max\Bigl\{0,1-\tfrac 2 R\dist\bigl(x,B(x_0,\tfrac R2)\bigr)\Bigr\}
\] 
for every $x\in X$. 
Then $\psi$ is a Lipschitz function in $X$ such that $\psi=0$ in $X\setminus B(x_0,R)$ and $\psi=1$ in $\overline{ B(x_0,R/2)}$. 
Define
\[
u(x)=\psi(x)(1-v(x))
\] 
for every $x\in X$. 
Then $u\colon X\to \R$ is continuous, $u=0$ in $E$ and $u=1-v$ in 
$\overline{B(x_0,R/2)}$.
In particular
\begin{equation}\label{e.samat}
\begin{split}
G_{u,\beta,q,A}(x)&=\biggl(\int_{A} \frac{\vert u(x)-u(y)\vert^q}{d(x,y)^{\beta q}\mu(B(x,d(x,y)))}\,d\mu(y)\biggr)^{1/q}\\
&=\biggl(\int_{A} \frac{\vert (1-v(x))-(1-v(y))\vert^q}{d(x,y)^{\beta q}\mu(B(x,d(x,y)))}\,d\mu(y)\biggr)^{1/q}=G_{v,\beta,q,A}(x)
\end{split}
\end{equation}
whenever $A\subset \overline{B(x_0,R/2)}$ is a Borel set and 
$x\in \overline{B(x_0,R/2)}$.

If $z\in F$, then $d(z,E)>0$ and we denote $B_z=B(z,2d(z,E))$.
Let $z\in F$ and observe that
$0<d(z,E)<\gamma R<R<\diam(E)/8$.
The pointwise Hardy inequality~\eqref{e.pw_hardy} implies
\[
\lvert u(z)\rvert^p\le c_H^p\dist(z,E)^{\beta p}  M_{2d(z,E)}\bigl((G_{u,\beta,q,B_{z}})^p\bigr)(z)\le  c_H^p R^{\beta p}M_{2d(z,E)}\bigl((G_{u,\beta,q,B_{z}})^p\bigr)(z)\,.
\]
Since $\lvert u(z)\rvert^p$ is finite, there exists  a radius
$0<r_z\le 2 \dist(z,E)<2\gamma R$ such that
\begin{equation}\label{e.choice}
\lvert u(z)\rvert^p\le 2 c_H^p R^{\beta p}\vint_{B(z,r_z)} G_{u,\beta,q,B_z}(x)^p\,d\mu(x)\,.
\end{equation}
Observe that the balls $B(z,r_z)$ form a cover of $F$ and  they have a uniformly bounded radii.

By the $5r$-covering lemma \cite[Lemma 1.7]{MR2867756}, there exist
pairwise disjoint balls $B(z_i,r_i)$, where $z_i\in F$ and
$r_i=r_{z_i}$ are as above, such that $F\subset\bigcup_{i=1}^{\infty}B(z_i,5r_i)$. 
It follows from~\eqref{e.mitta_F} and the doubling property of $\mu$ that
\begin{equation}\label{e.start_f}
\mu(B(x_0,R))\le \frac 2 {c_Q\gamma^Q} \mu(F) 
\le C(c_\mu)\sum_{i=1}^{\infty}\mu(B(z_i,r_i)).
\end{equation}
Let $i\in\N$. 
Since $z_i\in F\subset B(x_0,\gamma R)$ and $x_0\in E$, 
we have $\dist(z_i,E) < \gamma R$. 
If $x\in B(z_i,r_i)$, then 
\[
d(x_0,x)\le d(x_0,z_i) + d(z_i,x) 
\le \gamma R + 2\dist(z_i,E) < 3\gamma R= \frac R2,
\]
and thus $B(z_i,r_i)\subset B(x_0,R/2)$.  The same
argument shows that \[B_{z_i}=B(z_i,2d(z_i,E))\subset B(x_0,R/2)\,.\] Hence, by \eqref{e.samat}, we obtain
$G_{u,\beta,q,B_{z_i}}=G_{v,\beta,q,B_{z_i}}$  in $B(z_i,r_i)$.

Since $z_i\in F$, we have $u(z_i)=1-v(z_i)>\frac12$ and $0<d(z_i,E)<\gamma R<\diam(E)/8$.
By the choice of the radius $r_i$ in \eqref{e.choice}, we obtain
\begin{align*}
\Bigl(\frac12\Bigr)^p 
\le \lvert u(z_i)\rvert^p&\le 2
c_H^p R^{\beta p}\vint_{B(z_i,r_i)} G_{u,\beta,q,B_{z_i}}(x)^p\,d\mu(x)=2
c_H^p R^{\beta p}\vint_{B(z_i,r_i)} G_{v,\beta,q,B_{z_i}}(x)^p\,d\mu(x),
\end{align*}
and consequently
\begin{align*}
\mu(B(z_i,r_i))&\leq C(p,c_H) R^{\beta p} \int_{B(z_i,r_i)}G_{v,\beta,q,B_{z_i}}(x)^p\,d\mu(x)\\
&\le C(p,c_H) R^{\beta p} \int_{B(z_i,r_i)}G_{v,\beta,q,B(x_0,\Lambda R)}(x)^p\,d\mu(x)
\end{align*}
for every $i\in\N$.
Substituting these estimates into~\eqref{e.start_f} gives
\begin{align*}
\mu(B(x_0,R))
&\le C(p,c_\mu,c_H) R^{\beta p} \sum_{i=1}^{\infty} \int_{B(z_i,r_i)}G_{v,\beta,q,B(x_0,\Lambda R)}(x)^p\,d\mu(x)\\
&\le C(p,c_\mu,c_H) R^{\beta p} \int_{B(x_0,\Lambda R)} G_{v,\beta,q,B(x_0,\Lambda R)}(x)^p\,d\mu(x),
\end{align*}
where we also used the fact that the balls $B(z_i,r_i)\subset B(x_0,\Lambda R)$ are pairwise disjoint.
This shows that~\eqref{e.suffhere} holds, and the proof is complete.
\end{proof}

The following theorem characterizes the fractional capacity
density condition in connected metric spaces.
This result is a nonlocal or fractional analogue of \cite[Theorem 2]{MR2854110}.
To our knowledge, Theorem \ref{t.eqv_f} is new in this generality. 

\begin{theorem}\label{t.eqv_f}
Let $1\le p<\infty$, $1\le q<\infty$ and $0<\beta<1$. %
Assume that $X$ is connected. 
Let $E\subset X$ be a closed set.
Then the following conditions are equivalent:
\begin{itemize}
\item[(i)] $E$ satisfies the fractional $(\beta,p,q)$-capacity density condition~\eqref{e.fractionalk_capacity_density_condition}.
\item[(ii)] $E$ supports the boundary $(\beta,p,p,q)$-Poincar\'e inequality \eqref{eq.bdry_poinc_cap}.
\item[(iii)] $E$ supports the pointwise $(\beta,p,q)$-Hardy inequality \eqref{e.pw_hardy}.
\end{itemize}
\end{theorem}

\begin{proof}  
Since the metric space $X$ is connected, the measure  $\mu$ satisfies the quantitative reverse doubling condition~\eqref{e.reverse_doubling}, with
constants $\sigma>0$ and $c_\sigma>0$.
By Lemma \ref{l.qp_stronger}, we see that
$X$ supports the $(\beta,p,p,q)$-Poincar\'e inequality. Hence, the implication ~(i) to~(ii) follows from Lemma~\ref{l.bdry_poincare}.
Lemma~\ref{l.pointwise} shows that~(ii) implies~(iii).  The last 
 implication from (iii)  to~(i) follows
 from Lemma \ref{l.phtocd}. 
\end{proof}

 \begin{remark}
Suppose that a closed set $E\subset X$ supports the 
pointwise $(\beta,p,q)$-Hardy inequality~\eqref{e.pw_hardy}, for
some $1\le p,q<\infty$ and $0<\beta <1$. 
From   Definition \ref{d.bp_ph}(ii), it easily follows 
that $E$ supports the 
pointwise $(\hat \beta,\hat p,q)$-Hardy inequality if $p\le \hat p<\infty$ and $\beta\le \hat \beta<1$.
In the light of Theorem \ref{t.eqv_f}, a corresponding monotonicity result holds for 
the fractional capacity density condition in connected metric spaces, as well as for the boundary Poincar\'e inequality.
 Under certain restrictions,
the third parameter $q$ can be decreased, using Theorem \ref{l.q_monotonicity}.
Our main result, Theorem \ref{t.main_impro}, gives nontrivial extensions of the above mentioned three monotonicity properties in complete geodesic spaces.
\end{remark}

 \begin{remark}
The connectivity of the space $X$ is assumed in many results during this section. By examining the proofs, one could replace this assumption by requiring 
that, for all $0<\kappa<1$, there exists $0<c_R<1$ such that inequality \eqref{e.rev_dbl_decay} holds. This  holds in connected metric spaces by \cite[Lemma~3.7]{MR2867756}.
We omit  the corresponding slightly more general formulations.
\end{remark}

\section{Haj{\l}asz--Triebel--Lizorkin seminorms}\label{s.hajlasz}

One of our main goals is to show  
 a self-improvement property for 
the fractional capacity density condition~\eqref{e.fractionalk_capacity_density_condition}, see Theorem \ref{t.Riesz_and_Hajlasz}.
In order to show this result, we  make use of a self-improvement
property of a closely related Haj{\l}asz--Triebel-Lizorkin capacity density condition, which was recently established in  \cite{LMV}.
Our basic strategy of proof is to show that these two capacity density conditions are equivalent under certain assumptions.
For this purpose, we first recall local variants of Haj{\l}asz--Triebel--Lizorkin seminorms in this section. Later we need a Sobolev--Poincar\'e type inequality  in Lemma \ref{lemma:Sobolev-Poincare 2}. Its proof is recalled for convenience.

First  we
define fractional Haj{\l}asz gradients in open sets as in \cite{MR3471303}, by adapting the global definition in  \cite{MR2764899}.

\begin{definition}\label{d.fhg}
Let $0<\beta< 1$ and let $u\colon\Omega\to\R$  be a measurable function on an open set $\Omega\subset X$. A sequence $g=(g_k)_{k\in\Z}$ of nonnegative  measurable functions is called a fractional $\beta$-Haj{\l}asz gradient of $u$ in $\Omega$ if there exists
a set $E\subset \Omega$ such that $\mu(E)=0$ and for all 
$k\in \Z$ and $x,y\in \Omega\setminus E$ satisfying $2^{-k-1}\le d(x,y)<2^{-k}$, we have
\[
\lvert u(x)-u(y)\rvert \le d(x,y)^\beta(g_k(x)+g_k(y))\,.
\]
We denote by $\mathbb{D}^\beta_\Omega(u)$ the collection of all fractional
$\beta$-Haj{\l}asz gradients of $u$ in $\Omega$.
\end{definition}

Fix $1\le p<\infty$, $1\le q\le \infty$ and an open set $\Omega\subset X$. For a sequence $g=(g_k)_{k\in\Z}$ of measurable
functions in $\Omega$, we define
\[
\big\|g \big\|_{L^p(\Omega;\,l^q)}
=\big\|\|g\|_{l^q(\Z)}\big\|_{L^p(\Omega)}\,,
\]
where
\[
\|g\|_{l^q(\Z)}=
\begin{cases}
\big(\sum_{k\in\Z}|g_{k}|^{q}\big)^{1/q},&\quad\text{when }1\le q<\infty, \\
\;\sup_{k\in\Z}|f_{k}|,&\quad\text{when }q=\infty.
\end{cases}
\]
We use the following
Fefferman--Stein vector-valued maximal function inequality, we refer to \cite[Theorem 1.2]{MR2542655} for a proof.
If $1<p<\infty$ and $1< q\le \infty$, then there
exists a constant $C>0$ such that for all sequences  $(g_k)_{k\in\Z}$ of measurable
functions in $X$ we have
\begin{equation}\label{e.fs}
\lVert (Mg_k)_{k\in\Z}\rVert_{L^p(X;l^q)}\le C\lVert (g_k)_{k\in\Z}\rVert_{L^p(X;l^q)}\,,
\end{equation}
where $M g_k$ is the centered maximal function of $g_k$ for each $k\in\Z$,
see \eqref{e.max_funct_def_noncentered}.

\begin{definition}\label{d.htl}
Let $1<p<\infty$, $1\le q\le \infty$ and $0<\beta < 1$. The homogeneous Haj{\l}asz--Triebel--Lizorkin space $\dot{M}^\beta_{p,q}(\Omega)$ is the seminormed space of all measurable functions $u:\Omega\to\R$ on 
an open set $\Omega\subset X$ such that 
\[
\lvert u\rvert_{\dot{M}^{\beta}_{p,q}(\Omega)}=\inf_{g\in\mathbb{D}^\beta_\Omega(u)} \lVert g\rVert_{L^p(\Omega;l^q)}<\infty\,.
\]
\end{definition}

Notice that $\lvert u\rvert_{\dot{M}^{\beta}_{p,q}(\Omega)}\le \lvert u\rvert_{\dot{M}^{\beta}_{p,\hat q}(\Omega)}$ whenever  $1\le \hat q\le q\le \infty$.
We remark that $\lvert u\rvert_{\dot{M}^{\beta}_{p,q}(\Omega)}$ is a seminorm, but not a norm since it vanishes on constant functions. 

The following Sobolev--Poincar\'e inequality is from  \cite[Lemma 2.1]{MR3089750}. In the case $X=\R^n$, we refer to \cite[Theorem 2.3]{MR2764899}.

\begin{lemma}\label{lemma:Sobolev-Poincare 2}
Assume that the measure $\mu$ satisfies the quantitative doubling condition~\eqref{e.doubling_quant}, with exponent $Q>0$ and constant $c_Q>0$. 
Let $0<\beta <1$ and $0<t<Q/\beta$.
Then for every pair $\eps,\eps'\in (0,\beta)$ with $\eps<\eps'$, there exists
a constant $C=C(\eps,\eps',\beta,t,c_Q,Q)>0$ such that for all $x\in X$, $n\in\Z$, measurable functions $u$ in $B(x_0,2^{-n+1})$ and $(g_k)_{k\in\Z}\in \mathbb{D}_{B(x_0,2^{-n+1})}^\beta(u)$,
\begin{align*}
&\inf_{c\in\R}\left(\,\vint_{B(x_0,2^{-n})}|u(x)-c|^{t^*(\eps)}\,d\mu(x)\right)^{1/t^*(\eps)}\\
&\qquad \le C2^{-n\eps'}\sum_{j\ge n-2}2^{-j(\beta-\eps')}
\left(\,\vint_{B(x_0,2^{-n+1})}g_j(x)^{t}\,d\mu(x)\right)^{1/t},
\end{align*}
where $t^*(\eps)=Qt/(Q-\eps t)$.
\end{lemma}

\begin{proof}
For convenience, we remind reader of the proof from \cite[Lemma 2.1]{MR3089750}. It is 
 based  on a Sobolev--Poincar\'e inequality for Haj{\l}asz gradients, which
originally appeared in \cite{MR2039955}.  Below, we use the slightly more convenient
formulation appearing in \cite{MR4107806}.

Since $0<\varepsilon<1$ the function $d^\varepsilon$ defined by  $d^\varepsilon(x,y)=d(x,y)^\varepsilon$ for $x,y\in X$ is a metric in $X$.
The measure $\mu$ is doubling also with respect
to this metric and condition \eqref{e.doubling_quant} holds
in the metric measure space $(X,d^\varepsilon,\mu)$, with
exponent $Q/\varepsilon>Q/\beta$ and constant $c_Q>0$.  
We let  \[
h=\left(\sum_{j\ge n-2} 2^{-j(\beta-\varepsilon)t}g_j^t\right)^{1/t}\,.\]
Denote by $E$ the exceptional set for $(g_k)_{k\in\Z}\in \mathbb{D}_{B(x_0,2^{-n+1})}^\beta(u)$ as in Definition \ref{d.fhg}. 
Assume that $x,y\in B(x_0,2^{-n+1})\setminus E$ with $x\not=y$. Let $j\ge n-2$ be such that
$2^{-j-1}\le d(x,y)<2^{-j}$. Then
\begin{equation}\label{e.hajlasz}
\begin{split}
\lvert u(x)-u(y)\rvert &\le d(x,y)^\beta(g_j(x)+g_j(y))\\&\le d(x,y)^\varepsilon(h(x)+h(y))=d^\varepsilon(x,y)(h(x)+h(y))\,.
\end{split}
\end{equation}
This shows that the function $h$ is a so-called Haj{\l}asz-gradient of the function $u$ in the
ball $\{z\in X\,:\, d^\eps(z,x_0)<2^{-\eps(n+1)}\}$ of the metric
measure space $(X,d^\eps,\mu)$.
Therefore, we can apply \cite[Theorem 20 and Remark 21]{MR4107806}
in  $(X,d^\varepsilon,\mu)$ to obtain inequality
\begin{align*}
\inf_{c\in\R}\left(\,\vint_{B(x_0,2^{-n})}|u(x)-c|^{t^*(\eps)}\,d\mu(x)\right)^{1/t^*(\eps)}
\le C2^{-n\varepsilon} \left(\vint_{B(x_0,2^{-n+1})} h(x)^t\,d\mu(x)\right)^{1/t}\,.
\end{align*}
Here it is important to observe that $B(x_0,2^{-n})=\{y\in X\,:\,d^\varepsilon(y,x_0)<2^{-n\varepsilon}\}$.
Finally, it is straightforward
to show that
\begin{align*}
\left(\vint_{B(x_0,2^{-n+1})} h(x)^t\,d\mu(x)\right)^{1/t}
\le C 2^{-n(\varepsilon'-\varepsilon)}\sum_{j\ge n-2} 2^{-j(\beta-\varepsilon')}\left(\vint_{B(x_0,2^{-n+1})} g_j(y)^t\,d\mu(y)\right)^{1/t}\,,
\end{align*}
Indeed, if $0<t< 1$, then one proceeds as in \cite[Theorem 2.3]{MR2764899}. The
remaining case $t\ge 1$ is even easier, and this concludes the sketch of the proof. 
\end{proof}

\section{Seminorm comparison}\label{s.semi}

The following seminorm equivalence
\begin{equation}\label{e.besov}
C^{-1}\lvert u\rvert_{\dot M^{\beta}_{p,p}(X)}^p\le \int_{X}\int_{X}\frac{|u(x)-u(y)|^p}{d(x,y)^{\beta p}\mu(B(x,d(x,y)))}\,d\mu(y)\,d\mu(x)
\le C\lvert u\rvert_{\dot M^{\beta}_{p,p}(X)}^p\,,
\end{equation}
where $1\le p<\infty$ and $C$ is independent of $u$, 
is shown in \cite[Proposition 4.1]{MR2764899}. We establish 
a local variant of this equivalence, and also extend
it to the values of parameters $1\le q<\infty$ such that 
$\beta >Q(1/p-1/q)$. 
Here $Q$ is the exponent in the reverse doubling condition \eqref{e.doubling_quant}.  
Lemma \ref{l.stronger} and Lemma \ref{l.weaker} are the main results of this section; their proofs  are adaptations of the ideas from \cite[Proposition 4.1]{MR2764899}, combined with the use of the Fefferman--Stein maximal function inequality \eqref{e.fs}.
The cases of $\R^n$ and its uniform domains are treated in \cite{MR3667439}, where  sharpness of condition $\beta >Q(1/p-1/q)$ is discussed. We also refer to \cite{MR3089750,MR1097208,MR2250142}. 

We also use the following variant of a  discrete Young's inequality from \cite[Lemma 3.1]{MR3471303}.

\begin{lemma}\label{summing lemma}
Let $1<a<\infty$, $0<b<\infty$ and $c_k\ge 0$, $k\in\Z$. There is a constant $C=C(a,b)$ such that
\[
\sum_{k\in\Z}\Big(\sum_{j\in\Z}a^{-|j-k|}c_j\Big)^b\le C\sum_{j\in\Z}c_j^b.
\]
\end{lemma}

 The following lemma provides
a local counterpart of the second inequality in \eqref{e.besov}, also  extended
to exponents $1\le q<\infty$ satisfying $\beta >Q(\frac{1}{p}-\frac{1}{q})$.

\begin{lemma}\label{l.stronger}
Assume that the measure $\mu$ satisfies the quantitative doubling condition~\eqref{e.doubling_quant}, with exponent $Q>0$ and constant $c_Q>0$.
Let $0<\beta < 1$ and $1\le  p,q<\infty$ be such that $\beta >Q(\frac{1}{p}-\frac{1}{q})$.
Let $x_0\in X$ and $r>0$.
Then there exists a constant $C=C(\beta,p,q,c_\mu,c_Q,Q)>0$  such that
\[
\int_{B(x_0,r)}\left(\int_{B(x_0,r)}\frac{|u(x)-u(y)|^q}{d(x,y)^{\beta q}\mu(B(x,d(x,y)))}\,d\mu(y)\right)^{\frac{p}{q}}\,d\mu(x)\le C\lvert u\rvert_{\dot{M}^{\beta}_{p,q}(B(x_0,9 r))}^p
\]
for all  $u\in L^1(B(x_0,9 r))$.
\end{lemma}

\begin{proof}
We denote by $C$ a constant
that depends on $\beta$, $p$, $q$, $c_\mu$, $c_Q$ and $Q$ only.
We take $\lambda=9$.
Fix a Lebesgue point $x\in B(x_0,r)$ of $u$.
For each $j\in \Z$, we write $B_j(x)=B(x,2^{-j})$ and
\[A_j(x)=\{y\in B(x_0,r) :  2^{-j-1}\le d(x,y)<2^{-j}\}.\]
Choose $n\in\Z$ such that
$2^{-n-1}\le 2 r<2^{-n}$. 
We estimate
\begin{equation}\label{e.start}
\begin{split}
&\int_{B(x_0,r)}\frac{|u(x)-u(y)|^q}{d(x,y)^{\beta q}\mu(B(x,d(x,y)))}\,d\mu(y)\\
&\qquad =
\sum_{j=n}^\infty \int_{A_j(x)}\frac{|u(x)-u(y)|^q}{d(x,y)^{\beta q}\mu(B(x,d(x,y)))}\,d\mu(y)\\
&\qquad \le C(q,c_\mu)\sum_{j=n}^\infty 2^{j\beta q} \vint_{B_j(x)} |u(x)-u(y)|^q\,d\mu(y)\\
&\qquad \le C(q,c_\mu)\sum_{j=n}^\infty 2^{j\beta q}  |u(x)-u_{B_j(x)}\rvert^q + C(q,c_\mu)\sum_{j=n}^\infty 2^{j\beta q}\vint_{B_j(x)}  \vert u(y)-u_{B_j(x)}|^q\,d\mu(y)\,.
\end{split}
\end{equation}
Fix $g\in \mathbb{D}^\beta_{B(x_0,\lambda r)}(u)$
and define  $h_k=\mathbf{1}_{B(x_0,\lambda r)}  g_k$ for every $k\in\Z$. 
Observe that, for all $j\ge n$,  \[B(x,2^{-j+1})\subset B(x_0,9r)=B(x_0,\lambda r)\,.\] 
Let $Q>0$ and $c_Q>0$ be the constants in the quantitative doubling condition \eqref{e.doubling_quant}.
Since $\beta>Q(\frac{1}{p}-\frac{1}{q})$, there exists $0<t<Q/\beta$ and $0<\varepsilon<\varepsilon'<\beta$ such that $Q(\frac{1}{p}-\frac{1}{q})<\varepsilon$ and 
$Qt/(Q-\varepsilon t)=q$.
Observe that
\begin{equation}\label{e.exps}
0<t=\frac{qQ}{Q+q\varepsilon}<\min\{p,q\}\,.
\end{equation}
For every
 $j\ge n$, we apply  H\"older's inequality and Lemma \ref{lemma:Sobolev-Poincare 2} as follows
\begin{align*}
|u(x)-u_{B_j(x)}\rvert^q &\le \left(\sum_{k=0}^\infty \lvert u_{B_{j+k+1}(x)}-u_{B_{j+k}(x)}\rvert\right)^q\\
&\le C(c_\mu,q)\left(\sum_{k=0}^\infty \vint_{B_{j+k}(x)} \lvert u(y)-u_{B_{j+k}(x)}\rvert\,d\mu(y)\right)^q\\
&\le C(c_\mu,q)\left(\sum_{k=0}^\infty \left(\vint_{B_{j+k}(x)} \lvert u(y)-u_{B_{j+k}(x)}\rvert^q\,d\mu(y)\right)^{1/q}\right)^q\\
&\le C\left(\sum_{k=0}^\infty  2^{-(j+k)\eps'}\sum_{m\ge (j+k)-2}2^{-m(\beta-\eps')}
\left(\,\vint_{B_{j+k-1}(x)} h_m(y)^{t}\,d\mu(y)\right)^{1/t}
\right)^q\,.
\end{align*}
Hence, we find that
\begin{align*}
\sum_{j=n}^\infty2^{j\beta q}|u(x)-u_{B_j(x)}\rvert^q &\le C\sum_{j=n}^\infty2^{j\beta q}\left(\sum_{k=j}^\infty  2^{-k\eps'}\sum_{m\ge k-2}2^{-m(\beta-\eps')} (Mh_m^{t}(x))^{1/t}\right)^q\\
&\le C\sum_{j=n}^\infty\left(2^{j(\beta-\eps')}\sum_{m\ge j-2}2^{-m(\beta-\eps')} (Mh_m^{t}(x))^{1/t}\right)^q\\
&\le C\sum_{j=n}^\infty \left(\sum_{m\ge j-2}2^{-(m-j)(\beta-\eps')} (Mh_m^{t}(x))^{1/t}\right)^q\,.
\end{align*}
By Lemma \ref{summing lemma} with $a=2^{\beta-\eps'}>1$ and $b=q$, we obtain
\begin{equation}\label{e.first_h}
\begin{split}
\sum_{j=n}^\infty 2^{j\beta q}|u(x)-u_{B_j(x)}\rvert^q
&\le C\sum_{m=-\infty}^\infty(Mh_m^{t}(x))^{q/t}\,.
\end{split}
\end{equation}

Next we estimate the second term in the last line of  \eqref{e.start}.
By Lemma \ref{lemma:Sobolev-Poincare 2}, for every  $j\ge n$,
\begin{align*}
&2^{j\beta q}\vint_{B_j(x)}  \vert u(y)-u_{B_j(x)}|^q\,d\mu(y)\\
&\quad \le C2^{j\beta q}\left( 2^{-j\eps'}\sum_{k\ge j-2}2^{-k(\beta-\eps')}
\left(\,\vint_{B(x,2^{-j+1})} \mathbf{1}_{B(x_0,\lambda r)}(y)g_k(y)^{t}\,d\mu(y)\right)^{1/t} \right)^{q}\\
&\quad \le C2^{j\beta q}\left( 2^{-j\eps'}\sum_{k\ge j-2}2^{-k(\beta-\eps')}
\left(Mh_k^{t}(x)\right)^{1/t} \right)^{q}\\
&\quad \le C\left( \sum_{k\ge j-2}2^{-(k-j)(\beta-\eps')}
\left(Mh_k^{t}(x)\right)^{1/t} \right)^{q}\,.
\end{align*}
Again, by Lemma \ref{summing lemma} with $a=2^{(\beta-\varepsilon')}>1$ and $b=q$,  we get
\begin{equation}\label{e.second_g}
\begin{split}
&\sum_{j=n}^\infty 2^{j\beta q}\vint_{B_j(x)}  \vert u(y)-u_{B_j(x)}|^q\,d\mu(y)\\
&\quad \le C\sum_{j=n}^\infty\left( \sum_{k\ge j-2}2^{-(k-j)(\beta-\eps')}
\left(Mh_k^{t}(x)\right)^{1/t} \right)^{q}\le  C\sum_{k=-\infty}^\infty
\left(Mh_k^{t}(x)\right)^{q/t}\,.
\end{split}
\end{equation}

By combining \eqref{e.start}, \eqref{e.first_h} and \eqref{e.second_g}, we obtain
\begin{align*}
&\left(\int_{B(x_0,r)}\left(\int_{B(x_0,r)}\frac{|u(x)-u(y)|^q}{d(x,y)^{\beta q}\mu(B(x,d(x,y)))}\,d\mu(y)\right)^{\frac{p}{q}}\,d\mu(x)\right)^{1/p}\\
& \quad\le     
 C\left\lVert \left(\sum_{k=-\infty}^\infty
\left(Mh_k^{t}\right)^{q/t}\right)^{t/q}\right\rVert_{L^{p/t}(X)}^{1/t}\le     
C\left\lVert \left(\sum_{k=-\infty}^\infty
h_k^{q}\right)^{1/q}\right\rVert_{L^{p}(X)}\,,
\end{align*}
where the last step follows from \eqref{e.exps} and the Fefferman--Stein  inequality
\eqref{e.fs}. It remains to observe that 
\[
\left\lVert \left(\sum_{k=-\infty}^\infty
h_k^{q}\right)^{1/q}\right\rVert_{L^{p}(X)}\le C\left\lVert \left(\sum_{k=-\infty}^\infty
g_k^{q}\right)^{1/q}\right\rVert_{L^{p}(B(x_0,\lambda r))}\,.
\]
Hence, the proof is completed by taking infimum over all functions $g\in \mathbb{D}^\beta_{B(x_0,\lambda r)}(u)$.
\end{proof}

Next turn to a counterpart of the first inequality in \eqref{e.besov}.
We begin the following variant of the fractional Haj{\l}asz gradient  from \cite[p. 6]{MR2764899}. 
\begin{definition}\label{d.fhgtil}
	Let $\Omega\subset X$ be an open set, $0<\beta< 1$, $0<\varepsilon\le \beta$, $N\in\N$ and $u:\Om \to \R$ be a measurable function. We denote by ${\mathbb{D}}^{\beta,\varepsilon,N}_{\Om}(u)$ the collection of all sequences $g=(g_k)_{k\in\Z}$ of nonnegative measurable functions for which there exists a set $E$ with $\mu(E)=0$ such that, for all $x,y\in \Om\setminus E$, we have
\[
	|u(x)-u(y)|
	\leq d(x,y)^{\beta-\varepsilon}\sum_{j\in\Z} 2^{-j\varepsilon}\left(g_j(x)+g_j(y)\right)\mathbf{1}_{[2^{-j-N},\infty)}(d(x,y)).
\]
\end{definition}

We use the following two-sided comparison of the Haj{\l}asz--Triebel--Lizorkin seminorm, which is essentially from \cite[Theorem 2.1]{MR2764899}. For convenience of the reader
and completeness, we recall the proof below.

\begin{lemma}\label{lemma:varient_seminorm}
Let $\Omega\subset X$ be an open set,
$1\le p,q<\infty$, $0<\beta< 1$, $0<\varepsilon\le \beta$ and $N\in\N$. 
Assume that $u:\Om\subset X \to \R$ is a measurable function. Then there exists a constant $C=C(N,q,\varepsilon)>0$ such that
	\[
	\inf_{{g}\in{\mathbb{D}}^{\beta,\varepsilon,N}_{\Om}(u)}\|g\|_{L^p(\Om,\ell^q)}
	\le 
	 \inf_{{g}\in\mathbb{D}^\beta_{\Om}(u)}\|g\|_{L^p(\Om,\ell^q)}
	\leq C\inf_{{g}\in{\mathbb{D}}^{\beta,\varepsilon,N}_{\Om}(u)}\|g\|_{L^p(\Om,\ell^q)}\,.	
	\]
\end{lemma}

\begin{proof}
	Let ${g}\in\mathbb{D}^\beta_{\Om}(u)$ and let $E\subset \Omega$
	be the associated exceptional set of measure zero as in Definition \ref{d.fhg}. 
	Fix $x,y\in \Omega\setminus E$.
	Then we have for any $k\in\Z$,
	\[
	|u(x)-u(y)|\mathbf{1}_{[2^{-k-1},2^{-k})}(d(x,y))\leq d(x,y)^\beta \left(g_k(x)+g_k(y)\right) \mathbf{1}_{[2^{-k-1},2^{-k})}(d(x,y)).
	\]
	Taking sum, we have
	\begin{align*}
	|u(x)-u(y)|
	&=  \sum_{k\in\Z}  |u(x)-u(y)|\mathbf{1}_{[2^{-k-1},2^{-k})}(d(x,y))\\
	&\le  \sum_{k\in\Z} d(x,y)^\beta \left(g_k(x)+g_k(y)\right) \mathbf{1}_{[2^{-k-1},2^{-k})}(d(x,y))\\
	&\le   d(x,y)^{\beta-\varepsilon} \sum_{k\in\Z} 2^{-k\varepsilon} \left(g_k(x)+g_k(y)\right) \mathbf{1}_{[2^{-k-N},\infty)}(d(x,y)).
	\end{align*}
	Hence, we have ${g}\in{\mathbb{D}}^{\beta,\varepsilon,N}_{\Om}(u)$, so that $\mathbb{D}^\beta_{\Om}(u)\subset {\mathbb{D}}^{\beta,\varepsilon,N}_{\Om}(u)$. Consequently, we have 
	\[
	\inf_{{g}\in{\mathbb{D}}^{\beta,\varepsilon,N}_{\Om}(u)}\|g\|_{L^p(\Om,\ell^q)}
	\leq \inf_{{g}\in\mathbb{D}^\beta_{\Om}(u)}\|g\|_{L^p(\Om,\ell^q)}\,.
	\]
	
	Now we focus on the second inequality. 
	Let ${g}\in{\mathbb{D}}^{\beta,\varepsilon,N}_{\Om}(u)$
	 and let $E\subset \Omega$
	be the associated exceptional set of measure zero as in Definition \ref{d.fhgtil}. 
	 Define, for $k\in\Z$, 
	$
	h_k=\sum_{j=k-N}^\infty 2^{(k-j+1)\varepsilon}g_j.
	$
	Now, if $x,y\in \Omega\setminus E$ and $2^{-k-1}\le d(x,y)<2^{-k}$, we have 
	\begin{align*}
	|u(x)-u(y)|
	&\le d(x,y)^{\beta-\varepsilon}\sum_{j\in\Z} 2^{-j\varepsilon}\left(g_j(x)+g_j(y)\right)\mathbf{1}_{[2^{-j-N},\infty)}(d(x,y))\\
	&\le d(x,y)^{\beta-\varepsilon}\sum_{j=k-N}^\infty 2^{-j\varepsilon}\left(g_j(x)+g_j(y)\right)\\
	&\le d(x,y)^\beta\sum_{j=k-N}^\infty 2^{(k-j+1)\varepsilon}\left(g_j(x)+g_j(y)\right)= d(x,y)^\beta\left(h_k(x)+h_k(y)\right)\,.
	\end{align*}
	We have shown that ${h}=(h_k)_{k\in\Z}\in\mathbb{D}^\beta_{\Om}(u)$. By Lemma \ref{summing lemma} with $a=2^\varepsilon>1$ and $b=q$,
	\begin{align*}
	\|h\|_{\ell^q}^q
	&= \sum_{k\in\Z}\left(\sum_{j=k-N}^\infty 2^{(k-j+1)\varepsilon} g_j\right)^q\\
	&= 2^{(N+1)q\varepsilon}\sum_{k\in\Z}\left(\sum_{j=k-N}^\infty 2^{(k-(j+N))\varepsilon} g_{j+N-N}\right)^q\\
		&= 2^{(N+1)q\varepsilon}\sum_{k\in\Z}\left(\sum_{j=k}^\infty 2^{(k-j)\varepsilon} g_{j-N}\right)^q
	\le C(N,q,\varepsilon)\sum_{j\in\Z} g_{j-N}^q=C(N,q,\varepsilon)\lVert g\rVert_{\ell^q}^q.
	\end{align*}
	Hence 
	\[
	 \inf_{{H}\in\mathbb{D}^\beta_{\Om}(u)}\|H\|_{L^p(\Omega, \ell^q)}\le 
	\|{h}\|_{L^p(\Om,\ell^q)}
	\leq C(N,q,\varepsilon)\|{g}\|_{L^p(\Om,\ell^q)}\]
	for all ${g}\in{\mathbb{D}}^{\beta,\varepsilon,N}_{\Om}(u)$,
	 from which the second inequality follows.
\end{proof}

The following lemma provides
a local counterpart of the first inequality in \eqref{e.besov}, extended to all
values of the parameter $1\le q<\infty$.
The proof follows  \cite[Proposition 4.1]{MR2764899}
and we also make use of the Fefferman--Stein maximal function inequality \eqref{e.fs}.

\begin{lemma}\label{l.weaker}
Let $0<\beta< 1$ and  $1<p,q<\infty$.
Assume that  the measure $\mu$ satisfies the  
reverse doubling condition \eqref{e.rev_dbl_decay}, with constants
$\kappa=1/2$ and  $0<c_R<1$.
	Let  $x_0\in X$ and   $0<r<(1/64)\diam(X)$.
	Then  there exists a constant $C=C(c_\mu,c_R,p,q,\beta)>0$  such that
	\[
	\lvert u\rvert_{\dot{M}^{\beta}_{p,q}(B(x_0,r))}^p\leq C\int_{B(x_0,73r)}\left(\int_{B(x_0,73r)}\frac{|u(x)-u(y)|^q}{d(x,y)^{\beta q}\mu(B(x,d(x,y)))}\,d\mu(y)\right)^{\frac{p}{q}}\,d\mu(x)
	\]
	for all $u\in L^1(B(x_0,73r))$. 
\end{lemma}

\begin{proof}
We denote $B=B(x_0,r)$.
	Define, for every $j\in\mathbb{Z}$, the function $h_j\colon X\to [0,\infty]$, 
	\[
	h_j(x)= \vint_{B(x,2^{-j})}  g_j(z)\, d\mu(z)\,,\qquad x\in X\,,
	\]
	where,  for all $z\in X$, we define 
	\[
	g_j(z)= \mathbf{1}_{73 B}(z)\left(\int_{B(z,2^{-j+3})\setminus B(z,2^{-j})}\mathbf{1}_{73 B}(w)  \frac{|u(z)-u(w)|^q}{d(z,w)^{\beta q}\mu(B(z,d(z,w)))}\,d\mu(w)\right)^\frac{1}{q}\,.
	\]
	Fix $m\in\Z$ such that $2^{-m-1}\le 2 r<2^{-m}$.
  Fix Lebesgue points $x,y\in B$ of $u$, $x\not=y$, and choose $k\in\Z$ such that  $2^{-k-1}\leq d(x,y)<2^{-k}$. We observe that 
	  $k\ge m$. 

Let  $j\ge k-1$.
Then  $2^{-j+2}<\diam(X)/2$, and  
the assumed reverse doubling condition \eqref{e.rev_dbl_decay}, with constants $\kappa=1/2$ and $0<c_R<1$, implies
$\mu(B(x,2^{-j+1}))\le c_R\,\mu(B(x,2^{-j+2}))$.  Hence, we obtain
\begin{align*}
\mu(B(x,2^{-j+2})\setminus B(x,2^{-j+1}))\ge (1-c_R)\mu(B(x,2^{-j+2}))\,.
\end{align*}
	 For every $z\in B(x,2^{-j})$, we have $B(z,2^{-j+3})\subset 73B$ and 
	 \[
	B(x,2^{-j+2})\setminus B(x,2^{-j+1})\subset B(z,2^{-j+3})\setminus B(z,2^{-j})\,.
	\]
 By using the above estimates, we get 
\begin{align*}
\vint_{B(x,2^{-j})} &|u(z)-u_{B(x,2^{-j+2})\setminus B(x,2^{-j+1})}|\,	d\mu(z)\\
		&\leq \vint_{B(x,2^{-j})}\vint_{B(x,2^{-j+2})\setminus B(x,2^{-j+1})}|u(z)-u(w)|\, d\mu(w)\, d\mu(z)\\
		&\leq \vint_{B(x,2^{-j})}\left(\vint_{B(x,2^{-j+2})\setminus B(x,2^{-j+1})}|u(z)-u(w)|^q\,d\mu(w)\right)^\frac{1}{q}\,d\mu(z)\\
		&\leq C 2^{-j\beta}\vint_{B(x,2^{-j})}\left(\int_{B(x,2^{-j+2})\setminus B(x,2^{-j+1})}\frac{|u(z)-u(w)|^q}{d(z,w)^{\beta q}\mu(B(z,d(z,w)))}\,d\mu(w)\right)^\frac{1}{q}\, \,d\mu(z)\\
			&\leq C 2^{-j\beta}\vint_{B(x,2^{-j})}\left(\int_{B(z,2^{-j+3})\setminus B(z,2^{-j})}\frac{|u(z)-u(w)|^q}{d(z,w)^{\beta q}\mu(B(z,d(z,w)))}\,d\mu(w)\right)^\frac{1}{q}\, \,d\mu(z)\\
		&\leq C 2^{-j\beta}h_j(x)\,,
\end{align*}
where $C=C(c_\mu,c_R)$.
Hence, for each $j\ge k-1$, we obtain
\begin{equation}\label{eq-4}
\vint_{B(x,2^{-j})} |u(z)-u_{B(x,2^{-j+2})\setminus B(x,2^{-j+1})}|\,d\mu(z)\leq C(c_\mu,c_R) 2^{-j\beta}h_j(x)\,.
\end{equation}
	Now, using \eqref{eq-4}, we get
	\begin{equation}\label{eq-2}
	\begin{split}
		\sum_{j=k}^\infty &|u_{B(x,2^{-j})}-u_{B(x,2^{-j-1})}|\\
		 &\leq \sum_{j=k}^\infty \left(|u_{B(x,2^{-j})}-u_{B(x,2^{-j+2})\setminus B(x,2^{-j+1})}|
		+ |u_{B(x,2^{-j+2})\setminus B(x,2^{-j+1})}-u_{B(x,2^{-j-1})}|\right)\\
		 &\leq \sum_{j=k}^\infty \vint_{B(x,2^{-j})}|u(z)-u_{B(x,2^{-j+2})\setminus B(x,2^{-j+1})}|\,	d\mu(z)
		  \\&\qquad\qquad\qquad + \sum_{j=k}^\infty\vint_{B(x,2^{-j-1})}|u(z)-u_{B(x,2^{-j+2})\setminus B(x,2^{-j+1})}|	\, d\mu(z)\\
		 &\leq C(c_\mu)\sum_{j=k}^\infty\vint_{B(x,2^{-j})}|u(z)-u_{B(x,2^{-j+2})\setminus B(x,2^{-j+1})}|\,	d\mu(z)
		\leq C(c_\mu,c_R) \sum_{j=k}^\infty2^{-j\beta}h_j(x)
		\end{split}
	\end{equation}
	and similarly
	\begin{equation}\label{eq-3}
		\sum_{j=k+1}^\infty|u_{B(y,2^{-j})}-u_{B(y,2^{-j-1})}|
		\leq C(c_\mu,c_R) \sum_{j=k+1}^\infty2^{-j\beta}h_j(y).
	\end{equation}

Observe that $B(x,2^{-k})\subset B(y,2^{-k+1})\subset B(x,2^{-k+2})$ by the choice of $k$. 
Using  \eqref{eq-4} with $j=k-1$
and $y$ instead of $x$, we obtain
	\begin{equation}\label{eq-5}
	\begin{split}
		 |u_{B(x,2^{-k})}-u_{B(y,2^{-k-1})}|
		&\le  \vint_{B(x,2^{-k})}|u(z)-u_{B(y,2^{-k-1})}|\,d\mu(z)\\
	& \le C(c_\mu)\vint_{B(y,2^{-k+1})}|u(z)-u_{B(y,2^{-k-1})}|\,d\mu(z)\\
	& \leq C(c_\mu)\vint_{B(y,2^{-k+1})}|u(z)-u_{B(y,2^{-k+3})\setminus B(y,2^{-k+2})}|\,d\mu(z)\\
	&\qquad\qquad + C(c_\mu)|u_{B(y,2^{-k+3})\setminus B(y,2^{-k+2})}-u_{B(y,2^{-k-1})}|\\
	& \leq C(c_\mu)\vint_{B(y,2^{-k+1})}|u(z)-u_{B(y,2^{-k+3})\setminus B(y,2^{-k+2})}|\,d\mu(z)\\
	&\qquad\qquad + C(c_\mu)\vint_{B(y,2^{-k-1})}|u(z)-u_{B(y,2^{-k+3})\setminus B(y,2^{-k+2})}|\,d\mu(z)\\
&\leq C(c_\mu)\vint_{B(y,2^{-k+1})}|u(z)-u_{B(y,2^{-k+3})\setminus B(y,2^{-k+2})}|\,d\mu(z)\\
&\leq C(c_\mu,c_R) 2^{-(k-1)\beta}h_{k-1}(y).
\end{split}
\end{equation}
 Recall that $x$ and $y$ are both Lebesgue points of $u$. A chaining argument combined with inequalities \eqref{eq-2}, \eqref{eq-3} and \eqref{eq-5} gives 
	\begin{align*}
		|u(x)-u(y)|
		&\leq |u_{B(x,2^{-k})}-u_{B(y,2^{-k-1})}| \\&\qquad \qquad + \sum_{j=k}^\infty|u_{B(x,2^{-j})}-u_{B(x,2^{-j-1})}|
		+ \sum_{j=k+1}^\infty|u_{B(y,2^{-j})}-u_{B(y,2^{-j-1})}|\\
		&\leq C(c_\mu,c_R)\sum_{j=k-1}^\infty2^{-j\beta}(h_j(x)+h_j(y))\\
				&\leq C(c_\mu,c_R)\sum_{j\in\Z} 2^{-j\beta}(h_j(x)+h_j(y))\mathbf{1}_{[2^{-j-2},\infty)}(d(x,y))\,.
	\end{align*}
	That is $C(c_\mu,c_R)h=(C(c_\mu,c_R)h_j)_{j\in\Z}\in{\mathbb{D}}^{\beta,\beta,2}_{B}(u)$. Consequently Lemma \ref{lemma:varient_seminorm} gives
	\begin{equation}\label{eq-1}
		\lvert u\rvert_{\dot{M}^\beta_{p,q}(B)}^p
		\le  C(c_\mu,c_R,p,q,\beta)\|h\|_{L^{p}(B,\ell^{q})}^p
		\leq C(c_\mu,c_R,p,q,\beta)\int_{X}\left(\sum_{j\in\mathbb{Z}}h_j(x)^q\right)^\frac{p}{q}d\mu(x).
	\end{equation}
Observe that, for every $j\in\Z$,
	\[
	h_j(x)^q
	= \left(\vint_{B(x,2^{-j})}g_j(z)\,d\mu(z)\right)^q
	\leq Mg_j(x)^q.
	\]
	Hence \eqref{eq-1} and the Fefferman--Stein maximal 
	function inequality \eqref{e.fs} give
	\begin{align*}
		&\lvert u\rvert_{\dot{M}^\beta_{p,q}(B)}^p
		\leq
		C\left\|(Mg_j)_{j\in\Z}\right\|_{L^p\left(X,\ell^q\right)}^p
		\leq C\left\|(g_j)_{j\in\Z}\right\|_{L^p\left(X,\ell^q\right)}^p\\
		&= C \int_{73B}\left(\sum_{j\in\mathbb{Z}}\int_{B(z,2^{-j+3})\setminus B(z,2^{-j})} \mathbf{1}_{73B}(w)\frac{|u(z)-u(w)|^q}{d(z,w)^{\beta q}\mu(B(z,d(x,w)))}d\mu(w)\right)^\frac{p}{q}\,d\mu(z)\\
		&\le  C \int_{73B}\left(\int_{73B}\frac{|u(z)-u(w)|^q}{d(z,w)^{\beta q}\mu(B(z,d(x,w)))}\,d\mu(w)\right)^\frac{p}{q}\,d\mu(z)\,,
	\end{align*}
 where $C$ denotes a constant
that depends on the parameters $\beta$, $p$, $q$, $c_R$ and $c_\mu$ only. 
The desired inequality follows.
\end{proof}
\section{Comparison of relative capacities}\label{s.comparison_capacities}

In this section, we first define
a relative Haj{\l}asz--Triebel--Lizorkin capacity  by following \cite{LMV}, see Definition \ref{d.capacity}.
The seminorm comparison in Section \ref{s.semi} immediately yields 
a two-sided comparison between the relative fractional 
and Haj{\l}asz--Triebel--Lizorkin capacities, we refer to Theorem~\ref{l.second_side}.
As an application, we obtain
a $q$-monotonicity property of fractional capacities, see the second inequality in \eqref{e.cap_E_comp} and Theorem \ref{l.q_monotonicity}.
Under more restrictive assumptions, we also establish a deeper $q$-independence property
of fractional capacities, see
 Theorem \ref{l.q_independence}. This follows, via
 comparison of capacities, from the corresponding $q$-independence results in \cite{LMV}
 for relative Haj{\l}asz--Triebel--Lizorkin capacities.
 
\begin{definition}\label{d.capacity}
Let $1< p<\infty$, $1\le q\le \infty$, $0<\beta < 1$ and $\Lambda\ge  2$. 
Let $B\subset X$ be a ball and let $E\subset \iol{B}$ be a closed  set. Then we write
\[
\mathrm{cap}_{\dot{M}^{\beta}_{p,q}} (E,2B,\Lambda B) = \inf_\varphi \lvert \varphi\rvert_{\dot{M}^{\beta}_{p,q}(\Lambda B)}^p\,,
\]
where the infimum is taken over all continuous functions 
$\varphi\colon X\to \R$ such that $\varphi(x)\ge 1$ for every $x\in E$ and 
$\varphi(x)=0$ for every $x\in X\setminus 2B$.
\end{definition}

Global variants Haj{\l}asz--Triebel--Lizorkin capacities 
first appeared in~\cite{MR3605979} and 
they have been  studied, for instance, in~\cite{MR4104350}, \cite{KM} and \cite{math9212724}.
See also~\cite{MR1411441} and the references therein for  
the theory of Triebel--Lizorkin capacities in the Euclidean case.
 If $F\subset E$ are closed subsets of $\iol{B}$ and $1\le \hat q\le q\le \infty$, then
\begin{equation}\label{e.cap_E_comp}
\mathrm{cap}_{\dot{M}^{\beta}_{p,q}} (F,2B,\Lambda B)\le \mathrm{cap}_{\dot{M}^{\beta}_{p,q}} (E,2B,\Lambda B) \le \mathrm{cap}_{\dot{M}^{\beta}_{p,\hat q}} (E,2B,\Lambda B)\,.
\end{equation}

\begin{remark}\label{r.tfremark}
If $\varphi$ is a test function for $\mathrm{cap}_{\dot{M}^{\beta}_{p,q}} (E,2B,\Lambda B)$, then $v=\max\{0,\min\{\varphi,1\}\}$ is also a test function for the same relative capacity. Moreover,
\[\lvert v(x)-v(y)\rvert
\le \lvert \varphi(x)-\varphi(y)\rvert
\]
for all $x,y\in X$. 
It follows that $\lvert v\rvert_{\dot{M}^{\beta}_{p,q}(\Lambda B)}\le \lvert \varphi\rvert_{\dot{M}^{\beta}_{p,q}(\Lambda B)}$.
Hence, the infimum in Definition \ref{d.capacity} can be taken
either over {\em all continuous and bounded functions} or 
over {\em all continuous functions that are integrable on balls}, in both cases such that $\varphi(x)\ge 1$ for every $x\in E$ and 
$\varphi(x)=0$ for every $x\in X\setminus 2B$.
\end{remark}

The following theorem is nothing more than a capacitary counterpart
of the two-sided seminorm comparison proved in Section \ref{s.semi}.

\begin{theorem}\label{l.second_side}
Let $0<\beta < 1$, $1< p,q<\infty$ and $\Lambda\ge 2$.
Let $B=B(x_0,r)\subset X$ be a ball  and let $E\subset \iol{B}$ be a closed set. 
\begin{itemize}
\item[(i)]
Assume that $\mu$ satisfies the quantitative doubling condition~\eqref{e.doubling_quant}, with constants $Q>0$ and  $c_Q>0$. If $\beta >Q(\frac{1}{p}-\frac{1}{q})$, then
 there exists a constant $C=C(\beta,p,q,c_\mu,c_Q,Q)>0$ such that
\[
\cp_{\beta,p,q} (E,2B,\Lambda B)
\le C\cp_{\dot{M}^{\beta}_{p,q}} (E,2B,\Lambda_1 B)\,,
\]
where $\Lambda_1=9\Lambda$.

\item[(ii)]
Assume that  the measure $\mu$ satisfies the  
reverse doubling condition \eqref{e.rev_dbl_decay},  with constants
$\kappa=1/2$ and  $0<c_R<1$. 
If $\Lambda r < (1/64)\diam(X)$, then there exists a constant $C=C(c_\mu,c_R,p,q,\beta)>0$ such that
\[
\cp_{\dot{M}^{\beta}_{p,q}} (E,2B,\Lambda B)\le
C\cp_{\beta,p,q} (E,2B,\Lambda_2 B)\,,
\]
where $\Lambda_2=73\Lambda$.
\end{itemize}
\end{theorem}

\begin{proof}
This follows Lemma \ref{l.stronger} and Lemma \ref{l.weaker}, taking Remarks
\ref{r.fremark} and \ref{r.tfremark} into account.
\end{proof}

The Haj{\l}asz--Triebel--Lizorkin capacity, see Definition \ref{d.capacity}, is monotone in the $q$-parameter, see the second inequality in \eqref{e.cap_E_comp}.
By using the capacity comparison from Theorem \ref{l.second_side}, we now recover a similar property for
the fractional relative capacity.

\begin{theorem}\label{l.q_monotonicity}
Assume that the measure $\mu$ satisfies the quantitative doubling condition~\eqref{e.doubling_quant}, with exponent $Q>0$ and constant $c_Q>0$. 
 Let $\Lambda\ge 2$, $0<\beta< 1$,
$1<p<\infty$ and $1<\hat q\le q<\infty$  be such that $\beta > Q(\frac{1}{p}-\frac{1}{q})$.
Assume that the measure $\mu$ satisfies the  
reverse doubling condition~\eqref{e.rev_dbl_decay},  with constants
$\kappa=1/2$ and  $0<c_R<1$. Let $B=B(x_0,r)\subset X$ be a ball
such that $9\Lambda r<(1/64)\diam(X)$ and let $E\subset \iol{B}$ be a closed  set. 
Then there exists a constant 
$C=C(Q,c_Q,c_R,p,q,\hat q,\beta,c_\mu)>0$ such that
\[
\cp_{\beta,p,q} (E,2B,\Lambda B)
\le C\cp_{\beta,p,\hat q} (E,2B,\Lambda_3 B)\,.
\]
where $\Lambda_3=73\cdot 9 \Lambda$.
\end{theorem}

\begin{proof}
Theorem \ref{l.second_side}(A) implies that
\begin{align*}
\cp_{\beta,p,q} (E,2B,\Lambda B)
&\le C\cp_{\dot{M}^{\beta}_{p,q}} (E,2B,\Lambda_1 B)\,,
\end{align*}
where $\Lambda_1=9\Lambda$.
On the other hand, the trivial inequality \eqref{e.cap_E_comp} shows that 
\begin{equation}\label{e.trivial}
\cp_{\dot{M}^{\beta}_{p,q}} (E,2B,\Lambda_1 B)\le  \cp_{\dot{M}^{\beta}_{p,\hat q}} (E,2B,\Lambda_1 B)\,,
\end{equation}
Finally, we apply Theorem \ref{l.second_side}(B), which shows that
\[
\cp_{\dot{M}^{\beta}_{p,\hat q}} (E,2B,\Lambda_1 B)\le C\cp_{\beta,p,\hat q} (E,2B, \Lambda_3 B)\,,
\]
where $\Lambda_3=73\cdot 9 \Lambda$.
Combining the above estimates concludes the proof.
\end{proof}

The fractional capacity is, in fact,
independent of the $q$ parameter, aside from the constants of comparison. 
See the statement of  Theorem \ref{l.q_independence} for a more precise formulation.

\begin{example}
For the purpose of illustration, consider a connected metric space $X$.
Let $\Lambda>2$, $1\le p<\infty$ and $0<\beta <1$.
If $1\le q,\hat q<\infty$,  we
observe from parts (i) and (ii) of Lemma \ref{l.frac_cap_balls} that
\begin{align*}
(1/C) \mathrm{cap}_{\beta,p,q} (\overline{B},2B,\Lambda B)
\le r^{-\beta p}\mu(B)\le C\mathrm{cap}_{\beta,p,\hat{q}} (\overline{B},2B,\Lambda B)\,,
\end{align*}
where $B=B(x,r)$ 
is a ball in $X$ such that $0<r<(1/8)\diam(X)$.
The constant $C>1$ is independent of the ball $B$,
and therefore we see that the two different capacities,
associated with the independent parameters $q$ and $\hat q$, are comparable for all closed balls in $X$.
\end{example}

We extend the $q$-independence result of fractional capacities on closed balls to more general sets  under suitable assumptions, see Theorem \ref{l.q_independence}.
This deep result is via Theorem \ref{l.second_side} based on the corresponding
$q$-independence property \eqref{e.nontrivial}  of the relative Haj{\l}asz--Triebel--Lizorkin capacity. This property follows from \cite[Theorem
6.5, Theorem 9.3 and Theorem 6.6]{LMV}; we also refer \cite[Proposition 4.4.4]{MR1411441} and  \cite{MR1097178} for the Euclidean case.
 Observe also that the constant $\Lambda_3$ in \eqref{e.ss} is independent of $\Lambda$.

\begin{theorem}\label{l.q_independence}
Assume that the measure $\mu$ satisfies the quantitative doubling condition~\eqref{e.doubling_quant}, with exponent $Q>0$ and constant $c_Q>0$. 
 Let $\Lambda\ge 2$, $0<\beta< 1$,
$1<p<\infty$ and $1<q,\hat q<\infty$  be such that $\beta > Q(\frac{1}{p}-\frac{1}{q})$.
Assume that $\mu$ satisfies the  
reverse doubling condition~\eqref{e.rev_dbl_decay},  with constants
$\kappa=1/2$ and  $0<c_R<1$, and 
the quantitative reverse doubling condition~\eqref{e.reverse_doubling}, for
 some exponent $\sigma>\beta p$ and constant $c_\sigma>0$.  Let $B=B(x_0,r)\subset X$ be a ball
such that $0<r<(41\cdot 64)^{-1}\diam(X)$ and let $E\subset \iol{B}$ be a compact set. 
Then there exists a constant 
$C=C(Q,c_Q,c_\sigma,c_R,p,q,\hat q,\sigma,\beta,c_\mu,\Lambda)>0$ such that
\begin{equation}\label{e.ss}
\cp_{\beta,p,q} (E,2B,\Lambda B)
\le C\cp_{\beta,p,\hat q} (E,2B,\Lambda_3 B)\,.
\end{equation}
where $\Lambda_3=73\cdot 41$.
\end{theorem}

\begin{proof}
By  \cite[Theorem
6.5, Theorem 9.3 and Theorem 6.6]{LMV},
\begin{equation}\label{e.nontrivial}
\cp_{\dot{M}^{\beta}_{p,q}} (E,2B,\Lambda_1 B)\le C \cp_{\dot{M}^{\beta}_{p,\hat q}} (E,2B,41 B)\,,
\end{equation}
where $\Lambda_1=9\Lambda$.
Replacing the trivial inequality \eqref{e.trivial} by  \eqref{e.nontrivial}, we can otherwise argue as in the proof of Theorem \ref{l.q_monotonicity}, with obvious modifications.
\end{proof}

\section{Three ways of self-improvement}\label{s.three}

 In this section we prove  one of  our main result, Theorem \ref{t.main_impro}.
This result gives three ways of self-improvement for the fractional capacity density condition in a complete geodesic space: with respect to the parameters $\beta$, $p$ and $q$.
That is, if a closed set $E\subset X$ satisfies the fractional $(\beta,p,q)$-capacity density condition~\eqref{e.fractionalk_capacity_density_condition} with suitable restrictions on  $\beta$, $p$ and $q$, 
then there exists $\varepsilon>0$ such that
$E$ satisfies the fractional $(\hat \beta,\hat p,\hat q)$-capacity density condition~\eqref{e.fractionalk_capacity_density_condition}, if
\[
\lvert \beta-\hat \beta\rvert<\varepsilon\,,\qquad \lvert p-\hat p\rvert<\varepsilon\qquad\text{ and }\qquad 1\le \hat q<\infty\,.
\]

The proof of Theorem \ref{t.main_impro} is based on the analogous result for the  following Haj{\l}asz--Triebel--Lizorkin $(\beta,p,q)$-capacity density condition, proved in \cite{LMV}.
This result is taken  for granted in this paper, and 
our main result is established by creating a connection to it with the aid of the capacity
comparison in Section~\ref{s.comparison_capacities}. 
The basic estimates  in Section \ref{s.capacitary} are also useful here.

First, we need the following definition  from \cite[Definition 11.1]{LMV}.

\begin{definition}\label{d.HTL_cap_density}
 Let $1< p<\infty$, $1\le q\le\infty$, and $0<\beta<1$.
A closed set $E\subset X$ satisfies the {\em Haj{\l}asz--Triebel--Lizorkin $(\beta,p,q)$-capacity density condition}   
if there are constants  $c_0,c_1>0$ and $\Lambda>2$ such that
\begin{equation}\label{e.Hajlaszc capacity density condition}
\cp_{\dot{M}^\beta_{p,q}}(E\cap \overline{B(x,r)},B(x,2r),B(x,\Lambda r))\ge 
 c_0 \cp_{\dot{M}^\beta_{p,q}}(\overline{B(x,r)},B(x,2r),B(x,\Lambda r)) 
\end{equation}
for all $x\in E$ and all $0<r<c_1\diam(E)$.
\end{definition}

We also need a suitable version of the Hausdorff content density condition from \cite{MR3631460}; the following definition also coincides with 
 \cite[Definition 11.3]{LMV}.
 The reader is reminded to recall Definition \ref{d.hcc} at this stage.

\begin{definition}
A closed set $E\subset X$ satisfies the {\em Hausdorff content density condition of codimension $d\ge 0$} if  there is a constant $c_0>0$ such that
\begin{equation}\label{e.hausdorff_content_density}
\Ha^{\mu,d}_r(E\cap \overline{B(x,r)})\ge c_0 \: \Ha^{\mu,d}_r(\overline{B(x,r)})
\end{equation}
for   all   $x\in E$ and all $0<r<\diam(E)$.
\end{definition}

Next we show that a Hausdorff content
density condition,  for some $0 < d < \beta p$, implies the fractional $(\beta,p,q)$-capacity density condition.

\begin{theorem}\label{t.Hausdorff}
Assume that  the measure $\mu$ satisfies the  
reverse doubling condition \eqref{e.rev_dbl_decay}, with constants
$\kappa=1/2$ and  $0<c_R<1$. Let
$1\le p,q<\infty$ and $0<\beta< 1$.
Assume that a closed set $E\subset X$ 
 satisfies the Hausdorff content density condition \eqref{e.hausdorff_content_density} for some $0 < d < \beta p$.
Then  $E$ satisfies the fractional $(\beta,p,q)$-capacity density condition \eqref{e.fractionalk_capacity_density_condition}.
\end{theorem}

\begin{proof}
We adapt parts of the proof of  \cite[Theorem 11.4]{LMV}. We denote by $c_0>0$ the constant in the assumed inequality
\eqref{e.hausdorff_content_density}. 
We choose $\Lambda =4$, and fix $x\in E$ and $0<r<(1/8) \diam(E)$. 
Lemma~\ref{l.frac_cap_balls}(i) implies that 
\begin{equation}\label{e.b1}
\cp_{\beta,p,q}(\overline{B(x,r)}, B(x,2r),B(x,\Lambda r))
\le C(c_\mu,\beta,p,q) r^{-\beta p}\mu(B(x,r))\,.
\end{equation}
On the other hand, by  the simple estimate in~\cite[Remark 5.2]{CILV}
and \eqref{e.hausdorff_content_density}, 
we obtain
\begin{equation}\label{e.r_hausd}
r^{-d}\mu(B(x,r)) \le \Ha^{\mu,d}_{r}(\overline{B(x,r)})\le
C(c_0)\mathcal{H}^{\mu,d}_{r}(E\cap \overline{B(x,r)})\,.
\end{equation}
 Let $\{B(x_k,r_k)\}_k$ be a countable cover of $E\cap \overline{B(x,r)}$, satisfying 
  $B(x_k,r_k)\cap \overline{B(x,r)}\neq \emptyset$
and  $0<r_k\le 5\Lambda r$
 for all $k$. If $0<r_k\le r$ for all $k$, then by Definition \ref{d.hcc} of the Hausdorff content we have
\begin{equation*}%
\mathcal{H}^{\mu,d}_{r}(E\cap \overline{B(x,r)})\le 
\sum_{k} \mu(B(x_k,r_k))\,r_k^{-d}\,,
\end{equation*}
and thus, by~\eqref{e.r_hausd}, 
\begin{equation}\label{e.r_hausd_all}
r^{-d}\mu(B(x,r)) \le 
C(c_0)\sum_{k} \mu(B(x_k,r_k))\,r_k^{-d}\,.
\end{equation}
On the other hand, if $r<r_k$ for some $k$, then we may assume that $r<r_1\le 5\Lambda r$
and recall that $B(x_1,r_1)\cap \overline{B(x,r)}\neq \emptyset$. 
In this case we obtain, using also the doubling condition~\eqref{e.doubling}, that
\begin{equation}\label{e.r_hausd_one}
\begin{split}
r^{-d}\mu(B(x,r)) & \le C(d)r_1^{-d}\mu(B(x_1,3r_1))
\\&\le C(d,c_\mu)\mu(B(x_1,r_1))\,r_1^{-d}
 \le C(d,c_\mu)\sum_{k} \mu(B(x_k,r_k))\,r_k^{-d}\,.
\end{split}
\end{equation}
By taking infimum over all such covers of $E\cap \overline{B(x,r)}$, we conclude from
\eqref{e.r_hausd_all} and \eqref{e.r_hausd_one} that
\begin{equation}\label{e.r_hausd_bound}
r^{-d}\mu(B(x,r)) 
\le  C(c_0,d,c_\mu)\mathcal{H}^{\mu,d}_{5\Lambda r}(E\cap \overline{B(x,r)}).
\end{equation}

By assumption, the measure $\mu$ satisfies the reverse doubling condition~\eqref{e.rev_dbl_decay}, with constants $\kappa=1/2=2/\Lambda$
and $0<c_R<1$. 
By using Theorem~\ref{l.codim} with $\eta=d/\beta$,  we get
\begin{equation}\label{e.b2}
r^{-d}\mu(B(x,r)) \le C(\beta,q,p,d,c_R,c_\mu,c_0) r^{\beta p-d}\cp_{\beta,p,q}(E\cap \overline{B(x,r)},B(x,2r),B(x,\Lambda r))\,.
\end{equation}
Combining \eqref{e.b1} and \eqref{e.b2} gives
\[
\cp_{\beta,p,q}(\overline{B(x,r)}, B(x,2r),B(x,\Lambda r))\le
C \cp_{\beta,p,q}(E\cap \overline{B(x,r)},B(x,2r),B(x,\Lambda r))\,,
\]
where $C=C(\beta,q,p,d,c_R,c_\mu,c_0)$.  
We have shown that the set  $E$ satisfies the fractional $(\beta,p,q)$-capacity density condition \eqref{e.fractionalk_capacity_density_condition}.
\end{proof}

We are now ready to show the first version of our main result;
the crucial point to observe is that 
inequality $0<d<\beta p$ in 
condition (iii) of Theorem \ref{t.Riesz_and_Hajlasz} is open-ended.  Therefore, the two  other
equivalent conditions in the theorem are also open-ended or self-improving. 
The proof of the implication from (ii) to (iii) is significantly more difficult than proofs of the other implications. 
This deep implication relies on \cite[Theorem 11.4]{LMV} which, in turn, 
is via 
\cite[Theorem 6.4]{CILV}  based on \cite[Theorem 9.5]{MR4478471}.
 Recall that a doubling measure $\mu$ in a  geodesic space satisfies inequalities \eqref{e.doubling_quant} and  \eqref{e.reverse_doubling} for some exponents $Q>0$ and $\sigma>0$, respectively. 

\begin{theorem}\label{t.Riesz_and_Hajlasz}
 Let $(X,d,\mu)$ be a complete geodesic space  equipped  with a doubling measure $\mu$
which satisfies  the quantitative doubling condition~\eqref{e.doubling_quant},
 with exponent $Q>0$,
and  the quantitative reverse doubling condition \eqref{e.reverse_doubling}, with exponent $\sigma>0$.
 Let
$1<p,q<\infty$ and $0<\beta< 1$ be such that 
$Q(1/p-1/q)<\beta<\sigma/p$. 
 Then the following
conditions are equivalent for a closed set $E\subset X$:
\begin{enumerate}
\item[\textup{(i)}] $E$ satisfies the fractional $(\beta,p,q)$-capacity density condition \eqref{e.fractionalk_capacity_density_condition}.
\item[\textup{(ii)}] $E$ satisfies the Haj{\l}asz--Triebel--Lizorkin $(\beta,p,q)$-capacity density condition \eqref{e.Hajlaszc capacity density condition}, 
with some constants $\Lambda \ge 41$ and $c_1\le  1/80$.
\item[\textup{(iii)}] $E$ satisfies the Hausdorff content density condition \eqref{e.hausdorff_content_density} for some $0 < d < \beta p$.
\end{enumerate}
\end{theorem}

\begin{proof}
Assume that condition (i) holds.  Let $c_0>0$ and $\Lambda>2$ be the constant in  
\eqref{e.fractionalk_capacity_density_condition}.
Denote $\Lambda_1=\max\{41,9\Lambda\}$ and $\Lambda_2=\Lambda_1/9\ge \Lambda$.
 Fix $x\in E$ and $0<r<(1/80)\diam(E)$. 
By \cite[Theorem 4.3(a)]{LMV}, there is a constant $C_1=C(c_\mu,\beta,p,\Lambda)>0$ such that
\begin{equation}\label{e.r1}
\cp_{\dot{M}^{\beta}_{p,q}}(\overline{B(x,r)},B(x,2r),B(x,\Lambda_1 r)) \le C_1r^{-\beta p}\mu(B(x,r))\,.
\end{equation}
Since $X$ is geodesic, it is connected, and we can use
 Lemma~\ref{l.frac_cap_balls}(ii) to get
\begin{equation}\label{e.r2}
r^{-\beta p}\mu(B(x,r))
\le C(\beta,p, q, c_\mu,\Lambda)\mathrm{cap}_{\beta,p,q} (\overline{B(x,r)},B(x,2r),B(x,\Lambda r))\,.
\end{equation}
By combining \eqref{e.r1} and \eqref{e.r2}, we obtain
\begin{align*}
&\cp_{\dot{M}^{\beta}_{p,q}}(\overline{B(x,r)},B(x,2r),B(x,\Lambda_1 r))
\\&\qquad \le C(\beta,p,q, c_\mu,\Lambda)\mathrm{cap}_{\beta,p,q}(\overline{B(x,r)},B(x,2r),B(x,\Lambda r))\,.
\end{align*}
Since $E$ satisfies the fractional $(\beta,p,q)$-capacity density condition \eqref{e.fractionalk_capacity_density_condition}, with constants $c_0>0$ and $\Lambda>2$, and $\Lambda_2\ge  \Lambda$, we can continue to estimate
\begin{align*}
&\cp_{\dot{M}^{\beta}_{p,q}}(\overline{B(x,r)},B(x,2r),B(x,\Lambda_1 r))\\&\qquad \le C(\beta,p,q, c_\mu,c_0,\Lambda)\mathrm{cap}_{\beta,p,q}(E\cap \overline{B(x,r)},B(x,2r),B(x,\Lambda_2 r))\,.
\end{align*}
Theorem  \ref{l.second_side}(i) then  implies that
\begin{align*}
&\cp_{\dot{M}^{\beta}_{p,q}}(\overline{B(x,r)},B(x,2r),B(x,\Lambda_1 r))\\&\qquad \le C(\beta,p,q, c_\mu,c_0,c_Q,Q,\Lambda)\mathrm{cap}_{\dot{M}^{\beta}_{p,q}}(E\cap \overline{B(x,r)},B(x,2r),B(x,\Lambda_1 r))\,,
\end{align*}
for all $x\in E$ and $0<r<(1/80)\diam(E)$.
Therefore  (ii) holds.

The implication from (ii) to (iii) follows  from \cite[Theorem 11.4]{LMV}.
For the final implication we notice that, since $X$ is connected, the measure $\mu$ satisfies the reverse doubling condition~\eqref{e.rev_dbl_decay}, with constants $\kappa=1/2$
and $0<c_R=C(c_\mu)<1$.
Hence, the implication from (iii) to (i) follows from Theorem \ref{t.Hausdorff}.
\end{proof}

The following corollary explicates a three way self-improvement property of the fractional capacity density condition in complete geodesic spaces, under suitable restrictions on the parameters, most notable of which are the inequalities  $Q(1/p-1/q)<\beta<\sigma/p$. 

\begin{theorem}\label{t.main_impro}
Let $(X,d,\mu)$ be a complete geodesic space  equipped  with a doubling measure $\mu$
which satisfies  the quantitative doubling condition~\eqref{e.doubling_quant},
 with exponent $Q>0$,
and  the quantitative reverse doubling condition \eqref{e.reverse_doubling}, with exponent $\sigma>0$.
 Let
$1<p,q<\infty$ and $0<\beta< 1$ be such that 
$Q(1/p-1/q)<\beta<\sigma/p$.
Assume that a closed set $E\subset X$ satisfies the fractional $(\beta,p,q)$-capacity density condition~\eqref{e.fractionalk_capacity_density_condition}.
Then there exists $\varepsilon>0$ such that
$E$ satisfies also the fractional $(\hat \beta,\hat p,\hat q)$-capacity density condition~\eqref{e.fractionalk_capacity_density_condition}, if
\begin{equation}\label{e.eps}
\lvert \beta-\hat \beta\rvert<\varepsilon\,,\qquad \lvert p-\hat p\rvert<\varepsilon\qquad\text{ and }1\le \hat q<\infty\,.
\end{equation}
\end{theorem}

\begin{proof}
Theorem \ref{t.Riesz_and_Hajlasz} shows that $E$ satisfies the Hausdorff content density condition \eqref{e.hausdorff_content_density} for some $0 < d < \beta p$.
Choose  $\varepsilon>0$ such that the following inequalities
$1<\hat p<\infty$, $0<\hat \beta<1$, 
and
$0<d<\hat \beta \hat p$ hold for all 
$\hat \beta$, $\hat p$ satisfying the \eqref{e.eps}. This is clearly possible, since all
of the three inequalities are open-ended and the parameters
$\beta$ and $p$ satisfy them. 

Now, assume that \eqref{e.eps} holds for the chosen $\varepsilon>0$.
In particular, 
$E$ satisfies the Hausdorff content density condition \eqref{e.hausdorff_content_density} for  $0 < d < \hat \beta \hat p$.
Since $X$ is connected, the measure $\mu$ satisfies the reverse doubling condition~\eqref{e.rev_dbl_decay}, with constants $\kappa=1/2$
and $0<c_R=C(c_\mu)<1$.
The other assumptions of Theorem \ref{t.Hausdorff} are also satisfied, with parameters $\hat \beta$, $\hat p$ and $\hat q$, and the  theorem implies
that $E$ satisfies the fractional $(\hat \beta,\hat p,\hat q)$-capacity density condition~\eqref{e.fractionalk_capacity_density_condition}.
\end{proof}

\section{Fractional Hardy inequalities}\label{s.applications}

 A combination of Theorems \ref{t.main_impro} and \ref{t.eqv_f} provides 
three ways of self-improvement for a pointwise Hardy inequality in complete geodesic spaces, in terms of $\beta$, $p$ and $q$  satisfying  suitable restrictions. 
 Here we make use of this result  and show
that a pointwise $(\beta,p,q)$-Hardy inequality implies the following
variant of the integral $(\beta,p,q)$-Hardy inequality.

\begin{definition}\label{d.ih}
Let $1\le p,q<\infty$, $0<\beta<1$, and let $E\subset X$ be a closed set.
We say that  $E$ supports the ball $(\beta,p,q)$-Hardy  
inequality  if there are constants  $c_I>0$  and $\lambda\ge 1$ such that
\begin{equation}\label{eq.local_hardy_cap}
\begin{split}
\int_{B\setminus E} \frac{\lvert u(x)\rvert^p}{\dist(x,E)^{\beta p}}\,d\mu(x)\le c_I^p \int_{\lambda B} \biggl(\int_{\lambda B}\frac{|u(x)-u(y)|^q}{d(x,y)^{\beta q}\mu(B(x,d(x,y)))}\,d\mu(y)\biggr)^{p/q}\,d\mu(x)
\end{split}
\end{equation}
whenever $u\colon X\to \R$ is a continuous function 
such that %
$u=0$ on $E$ and $B=B(x_0,R)$ is a ball with
$x_0\in E$ and $0<R<\diam(E)/8$.
\end{definition}

Assuming that a closed set $E\subset X$ satisfies a pointwise $(\beta,p,q)$-Hardy inequality~\eqref{e.pw_hardy}, then
\[
\frac{\lvert u(x)\rvert^p}{\dist(x,E)^{\beta p}} \le c_H^p M_{2\dist(x,E)}\bigl((G_{u,\beta,q,B(x,2\dist(x,E))})^p\bigr)(x)\,,
\]
whenever $0<\dist(x,E)<\diam(E)/8$ and $u$ is as in \eqref{e.pw_hardy}.
Integration of this inequality does not readily  yield \eqref{eq.local_hardy_cap}, since
the restricted maximal operator is typically not bounded in  $L^1(X)$.  The following self-improvement of
the pointwise $(\beta,p,q)$-Hardy inequality in terms of the $p$-parameter allows us to apply a variant of this argument in the proof of Theorem \ref{t.integrated}.

\begin{theorem}\label{t.smaller}
 Let $(X,d,\mu)$ be a complete geodesic space  equipped  with a doubling measure $\mu$
which satisfies  the quantitative doubling condition~\eqref{e.doubling_quant},
 with exponent $Q>0$,
and  the quantitative reverse doubling condition \eqref{e.reverse_doubling}, with exponent $\sigma>0$.
 Let
$1<p,q<\infty$ and $0<\beta< 1$ be such that 
$Q(1/p-1/q)<\beta<\sigma/p$. 
If a closed set $E\subset X$ supports the pointwise $(\beta,p,q)$-Hardy inequality \eqref{e.pw_hardy}, then there exists an exponent $1<\hat p<p$  such that
$E$ supports the pointwise $(\beta,\hat p,q)$-Hardy inequality \eqref{e.pw_hardy}.
\end{theorem}

\begin{proof}
Theorem \ref{t.eqv_f} shows that the closed set $E$ satisfies the fractional $(\beta,p,q)$-capacity density condition~\eqref{e.fractionalk_capacity_density_condition}.
By Theorem \ref{t.main_impro}, there exists $1<\hat p<p$ such that  $E$ satisfies the fractional $(\beta,\hat p,q)$-capacity density condition \eqref{e.fractionalk_capacity_density_condition}. 
Now
$E$ supports also the pointwise $(\beta,\hat p,q)$-Hardy inequality \eqref{e.pw_hardy}, with a constant $c_H>0$, 
by Theorem \ref{t.eqv_f}.
\end{proof}

Let us emphasize that the implication from (i) to (ii)  is
the main novelty in Theorem \ref{t.integrated}.

\begin{theorem}\label{t.integrated}
 Let $(X,d,\mu)$ be a complete geodesic space  equipped  with a doubling measure $\mu$
which satisfies  the quantitative doubling condition~\eqref{e.doubling_quant},
 with exponent $Q>0$,
and  the quantitative reverse doubling condition \eqref{e.reverse_doubling}, with exponent $\sigma>0$.
 Let
$1<p,q<\infty$ and $0<\beta< 1$ be such that 
$Q(1/p-1/q)<\beta<\sigma/p$. 
 Then the following
conditions are equivalent for a closed set $E\subset X$:
\begin{itemize}
\item[(i)] $E$ supports the pointwise $(\beta,p,q)$-Hardy inequality \eqref{e.pw_hardy}.
\item[(ii)] $E$ supports the   ball  $(\beta,p,q)$-Hardy inequality~\eqref{eq.local_hardy_cap}.
\end{itemize}
\end{theorem}

\begin{proof}
 It is clear that condition (ii) implies that   $E$ supports the boundary  $(\beta,p,p,q)$-Poincar\'e inequality~\eqref{eq.bdry_poinc_cap}.
Since
$X$ is connected, the implication from (ii) to (i) follows from Theorem~\ref{t.eqv_f}.
It suffices to show the implication from (i) to (ii).
Assuming condition (i), we observe from Theorem \ref{t.smaller}  that there exists $1<\hat p<p$  such that
$E$ supports the pointwise $(\beta,\hat p,q)$-Hardy inequality \eqref{e.pw_hardy}.
By using this fact, we  show that \eqref{eq.local_hardy_cap} holds. 

Let $u\colon X\to \R$ be a continuous function 
such that %
$u=0$ on $E$.
Let $B=B(x_0,R)$ be a ball with
$x_0\in E$ and $0<R<\diam(E)/8$.
 Fix a point $x\in B\setminus E$. 
Then 
\[0<\dist(x,E)\le d(x,x_0)<R<\diam(E)/8\]
and 
$B(x,2\dist(x,E))\subset 3B$.
Therefore, by applying the pointwise $(\beta,\hat p,q)$-Hardy inequality~\eqref{e.pw_hardy}
and inequality \eqref{e.monotonicity},  we obtain
\begin{align*}
\frac{\lvert u(x)\rvert^p}{\dist(x,E)^{\beta p}} &\le  c_H^p
 \bigl(M_{2\dist(x,E)}\bigl((G_{u,\beta,q,B(x,2\dist(x,E))})^{\hat p}\bigr)(x)\bigr)^{p/{\hat p}}\\
& \le  c_H^p
 \bigl(M_{2\dist(x,E)}\bigl((G_{u,\beta,q,3B})^{\hat p}\bigr)(x)\bigr)^{p/{\hat p}} \le c_H^p
 \bigl(M\bigl(\mathbf{1}_{3B}(G_{u,\beta,q,3B})^{\hat p}\bigr)(x)\bigr)^{p/\hat p}.
\end{align*}
Hence, by integrating over the set $B\setminus E$, we obtain 
\[
\int_{B\setminus E} \frac{\lvert u(x)\rvert^p}{\dist(x,E)^{\beta p}}\,d\mu(x)
\le c_H^p\int_X \bigl(M\bigl(\mathbf{1}_{3B}(G_{u,\beta,q,3B})^{\hat p}\bigr)(x)\bigr)^{p/{\hat p}}\,d\mu(x).
\]
Since ${\hat p}<p$, the Hardy--Littlewood maximal function theorem~\cite[Theorem~3.13]{MR2867756}
implies 
\[
\int_{B\setminus E} \frac{\lvert u(x)\rvert^p}{\dist(x,E)^{\beta p}}\,d\mu(x)
\le C(c_H,{\hat p},p,c_\mu)\int_{3B} G_{u,\beta,q,3B}(x)^p\,d\mu(x).
\]
We have shown that the set $E$ supports
the  ball  $(\beta,p,q)$-Hardy inequality~\eqref{eq.local_hardy_cap}, with constants
$c_I=C(c_H,{\hat p},p,c_\mu)$ and $\lambda=3$.
\end{proof}

\def\cprime{$'$} \def\cprime{$'$} \def\cprime{$'$}

\setlength{\parindent}{0pt}

\end{document}